\documentclass[twoside,11pt]{article}

\usepackage{jmlr2e}

\usepackage{graphicx}
\usepackage[T1]{fontenc}    %
\usepackage{hyperref}       %
\usepackage{booktabs}       %
\usepackage{url}            %
\usepackage{amsfonts}       %
\usepackage{tikz}
\usepackage{pgfplots}
\usepackage{pgf}
\usepackage{epstopdf}
\usepgflibrary{shapes}
\usetikzlibrary{%
	arrows,%
	decorations.text,%
	positioning,%
	scopes,%
	shapes%
}
\usepackage{subfigure} 
\usepackage{hyperref}
\usepackage{algorithm,algorithmic}\usepackage{hyperref}
\usepackage{algorithm,algorithmic}

\usepackage{amssymb}
\usepackage{verbatim}
\usepackage{graphicx}					%
\usepackage{listings} 					%
\usepackage{amsmath}					%
\usepackage{amssymb}					%

\newcommand{\bT}{\mathbf{T}}
\newcommand{\R}{\mathbb{R}}
\newcommand{\Prob}{\mathbf{Pr}}
\newcommand{\E}{\mathbf{E}}
\DeclareMathOperator*{\argmax}{arg\,max}

\newtheorem{defi}{Definition}
\newtheorem{lem}{Lemma}
\newtheorem{thm}{Theorem}

\newtheorem{prp}{Proposition}

\newtheorem{assump}{Assumption}

\begin{document}

\title{Reinforcement Learning for Constrained and Multi-Objective Markov Decision Processes
}

\author{\name Ather Gattami \email ather.gattami@ai.se \\
	\addr AI Sweden\\
	Stockholm, Sweden \AND
	\name Qinbo Bai \email bai113@purdue.edu \\
	\addr School of Electrical and Computer Engineering\\
	Purdue University\\
	West Lafayette, IN 47907, USA
	\AND
	\name Vaneet Aggarwal\email vaneet@purdue.edu \\
	\addr School of IE and ECE\\
	Purdue University\\
	West Lafayette, IN 47907, USA}
		\editor{}

\maketitle

\begin{abstract}
In this paper, we consider the problem of optimization and learning for constrained and multi-objective Markov decision processes, for both discounted rewards and expected average rewards. We formulate the problems as zero-sum games where one player (the agent) solves a Markov decision problem and  its opponent solves a bandit optimization problem, which we here call Markov-Bandit games. We extend $Q$-learning to solve Markov-Bandit games and show that our new $Q$-learning algorithms converge to the optimal solutions of the zero-sum Markov-Bandit games, and hence converge to the optimal solutions of the constrained and multi-objective Markov decision problems. We provide a numerical example where we calculate the optimal policies and show by simulations that the algorithm converges to the calculated optimal policies. To the best of our knowledge, this is the first time learning algorithms guarantee convergence to optimal stationary policies for the constrained MDP problem with discounted and expected average rewards, respectively.

\end{abstract}

\section{Introduction}
\subsection{Motivation}
Reinforcement learning has made great advances in several applications, ranging from online learning and recommender engines, natural language understanding and generation, to mastering games such as Go \citep{silver:2017} and Chess. The idea is to learn from extensive experience how to take actions that maximize a given reward by interacting with the surrounding environment. The interaction teaches the agent how to maximize its reward without knowing the underlying dynamics of the process. A classical example is swinging up a pendulum in an upright position. By making several attempts to swing up a pendulum and balancing it, one might be able to learn the necessary forces that need to be applied in order to balance the pendulum without knowing the physical model behind it, which is the general approach of classical model based control theory \citep{astrom:1994}.

Informally, the problem of multi-objective reinforcement learning for Markov decision processes is described as follows.
Given a stochastic process with state $s_k$ at time step $k$, reward function $r^j$, and a discount factor $0<\gamma<1$, the multi-objective reinforcement learning problem is that for the optimizing agent to find a stationary policy $\pi(s_k)$ that satisfies for the discounted reward 
\begin{equation}
\label{dis}
\E\left( \sum_{k=0}^{\infty}\gamma^k r^j(s_k, \pi(s_k))\right) \ge 0
\end{equation}
or for the expected average reward
\begin{equation}
\label{ave}
\lim_{T\rightarrow \infty} \E\left(\frac{1}{T}\sum_{k=0}^{T-1} r^j(s_k, \pi(s_k))\right) \ge 0
\end{equation}
for  $ j=1, ..., J$ (a more formal definition of the problem is introduced in the next section).

The \textit{multi-objective} reinforcement learning problem of Markov decision processes is that of finding a policy that satisfies a number of constraints of the forms (\ref{dis}) or (\ref{ave}).

The following example from wireless communication describes in more detail a model where we have a Markov decision process with  constraints and where the agent doesn't have model knowledge.

\begin{example}[Altman, 1999]\label{eg:1}
	Consider a discrete time single-server queue with a buffer of finite size $L$. For a given time slot, we assume that at most one customer may join the system. The state of the system at a given time slot is the number of customers in the queue. There is a delay cost $c(s)$ given a state $s\in\{S_1, ..., S_n\}$ which one would like to keep as low as possible.  The probability of a service to be completed is $a^1$, where $1/a_1$ is the Quality of Service (QoS). The probability of queue arrival at time $t$ is $a^2$. The actions are given by $a^1$ and $a^2$. Let $c^1(a^1)$ be the cost to complete the service ($c^1$ is increasing in $a^1$). $c^1$ should be bounded by some value $v^1$. There is a cost corresponding to the throughput, $c^2(a^2)$,  ($c^2$ is decreasing in $a^2$). $c^2$ should be bounded by some value $v^2$.
	We assume that the number of actions is finite and actions sets are given by $a^1\in \{A_1^1, ..., A_{l_1}^1\}$ and $a^2\in \{A_1^2, ..., A_{l_2}^2\}$ where $0<A_1^1\le \cdots \le A_{l_1}^1\le 1$ and $0\le A_1^2 \le \cdots \le A_{l_2}^2 \le 1$. The transition probability $P(s_{k+1}, s_k, a^1_k, a^2_k)$ from state $s_k$ to $s_{k+1}$ given actions $a_k^1$ and $a_k^2$ is given by	
	\[ 
	P(s_+, s, a^1, a^2) =
	\begin{cases}
	(1-a^2)a^1      & \quad \text{if } L\ge s \ge 1,\\ &\quad s_+ = s-1\\
	a^2 a^1 + (1-a^2)(1-a^1)  & \quad \text{if } L\ge s \ge 1,\\ &\quad s_+ = s\\
	a^2 (1-a^1)      & \quad \text{if } L\ge s \ge 0,\\ &\quad s_+ = s+1\\
	1 - a^2 (1-a^1)     & \quad \text{if } L\ge s \ge 0,\\ &\quad s_+ = s = 0
	\end{cases}
	\]
	For $\gamma \in (0, 1)$, the constrained Markov decision process problem is given by
	\begin{equation}
	\label{serverex}
	\begin{aligned}
	\min_{\pi^1, \pi^2} ~~& \E\left( \sum_{k=0}^{\infty}\gamma^k c(s_k)\right)\\
	\textup{s. t.}~~ 	& \E\left( \sum_{k=0}^{\infty}\gamma^k c^1(\pi^1(s_k))\right) \le v^1 \\
	& \E\left( \sum_{k=0}^{\infty}\gamma^k c^2(\pi_2(s_k))\right) \le v^2,
	\end{aligned}
	\end{equation}
	which is equivalent to
	\begin{equation}
	\label{serverex}
	\begin{aligned}
	\max_{\pi^1, \pi^2} ~~& \E\left( \sum_{k=0}^{\infty}\gamma^k r(s_k, \pi(s_k)\right)\\
	\textup{s. t.}~~ 	& \E\left( \sum_{k=0}^{\infty}\gamma^k r^1(s_k, \pi(s_k))\right) \ge 0 \\
	& \E\left( \sum_{k=0}^{\infty}\gamma^k r^2(s_k, \pi(s_k))\right) \ge 0,
	\end{aligned}
	\end{equation}
	where $a_k = (a^1_k, a^2_k)$, $\pi(s_k) = (\pi^1(s_k), \pi^2(s_k))$, $r(s_k, a_k) = -c(s_k)$, $r^1(s_k, a_k) = -c^1(a^1_k) + v^1\cdot(1-\gamma)$,
	$r^2(s_k, a_k) = -c^2(a^2_k) + v^2\cdot (1-\gamma)$
\end{example}

\begin{example}[Search Engine]
	In a search engine, there is a number of documents that are related to a certain query. There are two values that are related to every document, the first being a (advertisement) value $u_i$ of document $i$ for the search engine and the second being a value $v_i$ for the user (could be a measure of how strongly related the document is to the user query). The task of the search engine is to display the documents in a row some order, where each row has an attention value, $A_j$ for row $j$. We assume that $u_i$ and $v_i$ are known to the search engine for all $i$, whereas the attention values $\{A_j\}$ are not known.
	The strategy $\pi$ of the search engine is to display document $i$ in position $j$, $\pi(i) = j$, with probability $p_{ij}$. Thus, the expected average reward for the search engine is
	$$
	R^{e} = \lim_{N\rightarrow \infty} \frac{1}{N}\sum_{i=1}^N \E(u_i A_{\pi(i)}) 
	$$
	and for the user
	$$
	R^u = \lim_{N\rightarrow \infty} \frac{1}{N}\sum_{i=1}^N \E(v_i A_{\pi(i)}). 
	$$
	The search engine has multiple objectives here where it wants to maximize the rewards for the user and itself. One solution is to define a measure for the quality of service for the user, $R^u\ge \underline{R}^u$ and at the same time satisfy a certain lower bound $\underline{R}^e$ of its own reward, that is
	\[
	\begin{aligned}
	\textup{find }   ~ &~ \pi \\
	\textup{s. t.}	~ &\lim_{N\rightarrow \infty} \frac{1}{N}\sum_{i=1}^N \E(u_i A_{\pi(i)}) \ge \underline{R}^e\\
	~ &  \lim_{N\rightarrow \infty} \frac{1}{N}\sum_{i=1}^N \E(v_i A_{\pi(i)}) \ge \underline{R}^u
	\end{aligned}
	\]
	
\end{example}

Surprisingly, although multi-objective (subclass of constrained) MDP problems are fundamental and have been studied extensively in the literature (see \citep{Altman99constrainedmarkov} and the references therein), the reinforcement learning counter part  seem to be still open. When an agent have to take actions based on the observed states, and constraint-outputs solely (without any knowledge about the dynamics, and/or constraint-functions), a general solution seem to be lacking to the best of the author's knowledge for both the discounted and expected average rewards cases. 

Note that maximizing (\ref{dis}) is equivalent to maximizing $\delta$ subject to the constraint
\begin{equation*}
\E\left( \sum_{k=0}^{\infty}\gamma^k r(s_k, \pi(s_k))\right) \geq \delta
\end{equation*}
which is in turn equivalent to
\begin{equation*}
\E\left( \sum_{k=0}^{\infty}\gamma^k (r(s_k, \pi(s_k)) - (1-\gamma)\delta)\right) \geq 0
\end{equation*}
since 
$$
\sum_{k=0}^{\infty}\gamma^k (1-\gamma)\delta = \frac{1}{1-\gamma}\cdot (1-\gamma)\delta = \delta.
$$
Thus, one could always replace $r$ with $r - (1-\gamma)\delta$ and obtain a constraint of the form (\ref{dis}). Similarly for the average reward case, one may replace $r$ with $r-\delta$ to obtain a constraint of the form  (\ref{ave}). Hence, we can run the bisection method with respect to $\delta$. 
In this paper, we will consider the problem of finding a policy that simultaneously satisfies constraints of the form (\ref{dis}) or (\ref{ave}).
In certain application, one might want to impose constraints on some reward functions for each time step. For the wireless communication example, there are certain applications in the 5G network architecture were high reliability/low latency requirements impose strict delay constraints at every time step \citep{johansson:2015}. That is, we will need more strict  constraints of the form
\begin{equation*}
\E(r^j(s_k, \pi(s_k))) \ge 0, ~~~\textup{for all } k, ~~~  j=1, ..., J-1,
\end{equation*}
where the expectation is taken with respect to $\pi$ 
(a more formal definition of the problem is introduced in the next section). 

\subsection{Previous Work}

Constrained MDP problems are convex and hence one can convert the constrained MDP problem to an unconstrained zero-sum game where the objective is the Lagrangian of the optimization problem \citep{Altman99constrainedmarkov}.  However, when the dynamics and rewards are not known, it doesn't become apparent how to do it as the Lagrangian will itself become unkown to the optimizing agent. 
Previous work regarding constrained MDPs, when the dynamics of the stochastic process are not known, considers scalarization through weighted sums of the rewards, see \citep{Roijers:2013} and the references therein. 
Another approach is to consider Pareto optimality when multiple objectives are present \citep{JMLR:v15:vanmoffaert14a}.
However, none of the aforementioned approaches guarantee to satisfy lower bounds for a given set of reward functions simultaneously.  Further, we note that  deterministic policies are not optimal \citep{Altman99constrainedmarkov}. 

In \citep{geibel:2006}, the author considers a single constraint and allowing for randomized policies. However,  no proofs of convergence are provided for the proposed sub-optimal algorithms. Sub-optimal solutions with convergence guarantees were provided in \citep{chow:2017} for the single constraint problem, allowing for randomized polices. In \citep{borkar:2005}, an actor-critic sub-optimal algorithm is provided for one single constraints and it's claimed that it can generalized to an arbitrary number of constraints. Reinforcement learning based model-free solutions have been proposed for the problems without guarantees \citep{djonin:2007,lizotte:2010,drugan:2013,achiam:2017,abels:2019,raghu2019deep}. 

Recently, \citep{RCPO} proposed a policy gradient algorithm with Lagrange multiplier in multi-time scale for discounted constrained reinforcement learning algorithm and proved that the policy converges to a feasible policy. \citep{EECM} found a feasible policy by using Lagrange multiplier and zero-sum game for reinforcement learning algorithm with convex constraints and discounted reward. \citep{paternain2019constrained} showed that constrained reinforcement learning has zero duality gap, which provides a theoretical guarantee to policy gradient algorithms in the dual domain. In constrast, our paper does not use policy gradient based algorithms.  \citep{CUCRL} proposed the C-UCRL algorithm which achieve sub-linear $O(T^{\frac{3}{4}}\sqrt{\log(T)/\delta})$ with probability $1-\delta$, while satisfiying the constraints. However, this algorithm needs the knowledge of the model dynamics. \cite{EECM} proposed 4 algorithms for the constrained reinforcement learning problem in primal, dual or primal-dual domain and showed a sub-linear bound for regret and constraints violations. However, all  these algorithms are model based. 

\subsection{Contributions}
We consider the problem of optimization and learning for constrained and multi-objective Markov decision processes, for both discounted rewards and expected average rewards. We formulate the problems as zero-sum games where one player (the agent) solves a Markov decision problem and  its opponent solves a bandit optimization problem, which we here call Markov-Bandit games which are interesting on their own. The opponent acts on a \textit{finite} set (and not on a continuous space). This transformation is essential in order to achieve a tractable optimal algorithm. The reason is that using Lagrange duality without model knowledge requires infinite dimensional optimization in the learning algorithm since the Lagrange multipliers are continuous (compare to the intractability of a partially observable MDP, where the beliefs are continuous variables).We extend $Q$-learning to solve Markov-Bandit games and show that our new $Q$-learning algorithms converge to the optimal solutions of the zero-sum Markov-Bandit games, and hence converge to the optimal solutions of the constrained and multi-objective Markov decision problems. The proof techniques are different for solving the discounted and average rewards problems, respectively, where the latter becomes much more technically involved.
We provide a numerical example where we calculate the optimal policies and show by simulations that the algorithm converges to the calculated optimal policies. 
To the best of our knowledge, this is the first time learning algorithms guarantee convergence to optimal stationary policies for the constrained and multi-objective MDP problem with discounted and expected average rewards, respectively.

\subsection{Notation}
{
	\begin{tabular}{ll}
		$\mathbb{N}$ & The set of nonnegative integers.\\
		$[J]$ & The set of integers $\{1, ..., J\}$.\\
		$\mathbb{R}$ & The set of real numbers.\\
		$\mathbf{E}$ & The expectation operator.\\
		$\mathbf{Pr}$ &  $\mathbf{Pr}(x\mid y)$ denotes the probability of the\\
		&  stochastic variable $x$ given $y$.\\
		$\argmax$ 		& $\pi^\star = \argmax_{\pi\in \Pi} f_\pi$ denotes an 			
		element\\
		&  $\pi^\star \in 	\Pi$ that maximizes the function $f_\pi$.\\
		$\arg \max \min$ 		& $\pi^\star = \arg \max_{\pi\in \Pi} \min f_{\pi,o}$ denotes an 			
		element\\
		&  $(\pi^\star, o^\star) \in \Pi\times O$ that takes the maxmin over $f_{\pi,o}$.\\
		$\ge$ 				& For $\lambda = (\lambda_1, ..., \lambda_{J})$, 
		$\lambda \ge 0$ denotes\\ 
		& that $\lambda_i\ge 0$ for $i=1, ..., J$.	\\
		$1_{X}(x)$ & $1_{X}(x) = 1$ if $x\in \{X\}$ and\\ 
		& $1_{X}(x) = 0$ if $x\notin \{X\}$.\\
		$\mathbf{1}_n$ & $\mathbf{1}_{n} = (1, 1, ..., 1)\in \R^n$. \\
		$N(t, s, a, b)$ & $N(t, s, a, b) = \sum_{k=1}^t 1_{(s, a, b)}(s_k, a_k, b_k)$. \\
		e & $\textup{e}: (s,a,o)\mapsto 1$.\\
		$|S|$ & Denotes the number of elements in $S$.\\
		$s_+$        & For a state $s = s_k$, we have $s_+ = s_{k+1}$.
	\end{tabular}
}
\subsection{Outline}
In the problem formulation (Section \ref{probform}), we present a precise mathematical definition of the constrained reinforcement learning problem for MDPs. Then, we give a brief introduction to reinforcement learning with applications to zero-sum games and some useful results in the section on reinforcement learning for zero-sum Markov games (Section \ref{0-sum-games}). A solution to the constrained reinforcement learning problem is then presented in Section \ref{game-approach}. We demonstrate the proposed algorithm by  examples in Section \ref{sim} and we finally conclude the paper and  discuss future work in Section \ref{conc}. Most of the proofs are relegated to the Appendix.

\section{Problem Formulation}
\label{probform}
Consider a Markov Decision Process (MDP) defined by the tuple $(S, A, P)$, where  $S = \{S_1, S_2, ..., S_n\}$ is a finite set of states, $A = \{A_1, A_2, ..., A_m\}$ is a finite set of actions taken by the agent, and $P:S\times A \times S \rightarrow [0,1]$ is a transition function mapping each triple $(s, a, s_+)$ to a probability given by
$$P(s, a, s_+)=\mathbf{Pr}(s_+ \mid s, a)$$
and hence,
$$
\sum_{s_+\in S} P(s, a, s_+) = 1, ~~~ \forall (s,a)\in S\times A.
$$
Let $\Pi$ be the set of policies that map 
a state $s\in S$ to a probability distribution of the actions with a probability assigned to each action $a\in A$, that is $\pi(s) = a$ with probability $\Prob(a\mid s)$.
The agent's objective is to find a stationary policy $\pi\in \Pi$ that maximizes the expected value of the total discounted reward
or the expected value of the average reward 
for $s_0 = s\in S$, for some possibly unknown reward function.%

Multi-objective reinforcement learning is concerned with finding a policy that satisfies a set of constraints of the form (\ref{dis}) or (\ref{ave}), 
where $r^j:S\times A \rightarrow \R$ are bounded functions, for $j=1, ..., J$, possibly unknown to the agent. The parameter $\gamma \in (0,1)$ is a discount factor which models how much weight to put on future rewards. The expectation is taken with respect to the randomness introduced by the policy $\pi$ and the transition mapping $P$.  

\begin{defi}[Unichain MDP]
	An MDP is called unichain, if for each policy $\pi$ the Markov chain
	induced by $\pi$ is ergodic, i.e. each state is reachable from any other state.
\end{defi}
Unichain MDP:s are usually considered in reinforcement learning problems with discounted rewards,  since they guarantee that we learn the process dynamics from the initial states. Thus, for the discounted reward case we will make the following assumption. 
\begin{assump}[Unichain MDP]
	\label{unichain}
	The MDP $(S, A, P)$ is assumed to be unichain.
\end{assump}
For the case of expected average reward, we will make a simpler assumption regarding the existence of a recurring state, a standard assumption in Markov decision process problems with expected average rewards to ensure that the expected reward is independent of the initial state.

\begin{assump}
	\label{recurrentstate}
	There exists a state $s^* \in S$ which is recurrent for every stationary policy $\pi$ played by the agent.
\end{assump}

Assumption \ref{recurrentstate} implies that $\E(r^j(s_k, a_k))$ is independent of the initial state at stationarity. Hence, the constraint %
(\ref{ave}) is at stationarity equivalent to the inequality $\E(r^j(s_k, a_k)) \ge 0$, for all $k$.
We will use this constraint in the sequel which turns out to be very useful in the game-theoretic approach to solve the problem.
\begin{assump}
	\label{r}
	The absolute values of the reward functions $r$ and $\{r^j\}_{j=1}^{J}$ are bounded by some constant $c$ known to the agent.
\end{assump}

\section{Reinforcement Learning for Zero-Sum Markov-Bandit Games}
\label{0-sum-games}

A zero-sum Markov-Bandit game is defined by the tuple $(S, A, O, P, R)$, where $S$, $A$ and $P$ are defined as in section \ref{probform}, $O=\{o_1, o_2, ..., o_q\}$ is a finite set of actions made by the agent's \textit{opponent}. 
Let $\Pi$ be the set of policies $\pi(s)$ that map a state $s\in S$ to a probability distribution of the actions with a probability assigned to each action $a\in A$, that is $\pi(s) = a$ with probability $\Prob(a\mid s)$. 

For the zero-sum Markov-Bandit game, we define the reward $R: S\times A \times O \rightarrow \R$ which is assumed to be bounded. The agent's objective is to maximize the minimum (average or discounted) reward obtained due to the opponent's malicious action. The difference between a zero-sum Markov game and a Markov-Bandit game is that the opponent's action doesn't affect the state and it chooses a constant action $o_k=o\in O $ for all time steps $k$. This will be made more precise in the following sections.

\subsection{Discounted Rewards}
Consider a zero-sum Markov-Bandit game where the agent is maximizing the total discounted reward given by
\begin{equation}
	\label{minreward}
	V(s) = \min_{o}\E\left( \sum_{k=0}^{\infty}\gamma^k R(s_k, a_k, o)\right) 
\end{equation} 
for the initial state $s_0\in S$. Let $Q(s,a,o)$ be the expected reward of the agent taking action $a_0=a\in A$ from state $s_0=s$, and continuing with a policy $\pi$ thereafter when the opponent takes a fixed action $o$.
Note that this is different from zero-sum Markov games with discounted rewards \citep{Littman:1994}, where the opponent's actions may vary over time, that is $o_k$ is not a constant. 
Then for any stationary policy $\pi$, we have that
\begin{equation}
	\label{Q} 
	\begin{aligned}
		Q(s,a,o) 	&= R(s, a, o) + 
		\E\left( \sum_{k=1}^{\infty}\gamma^k R(s_+, \pi(s_+), o)\right)\\
		& = R(s, a, o) + \gamma \cdot 
		\E\left(Q(s_+, \pi(s_+), o)\right)\\
	\end{aligned}
\end{equation} 

Equation (\ref{Q}) is known as the Bellman equation. The solution to (\ref{Q}), with respect to $Q$ and the initial state $s_0$ that corresponds to the optimal policy $\pi^\star$, is denoted $Q^\star$.
If we have the function $Q^\star$, then we can obtain the optimal policy $\pi^\star$ according to the equations
\begin{equation}
	\label{pistar0}
	\begin{aligned}
		&Q^\star(s,a,o) = 
		R(s, a, o) + \gamma \cdot 
		\E\left(Q^\star(s_+, \pi^\star(s_+), o)\right)\\
		&(\pi^\star(s_0), o^\star) = \arg \max_{\pi\in \Pi} \min_{o} \E\left(Q^\star(s_0, \pi(s_0), o)\right) \\
		&\pi^\star(s) = \argmax_{\pi\in \Pi} \E\left(Q^\star(s, \pi(s), o^\star)\right) \\
	\end{aligned}
\end{equation}
which maximizes the total discounted reward 
{\small
	\begin{equation*}
		\begin{aligned}
			\min_{o}\E\left( \sum_{k=0}^{\infty}\gamma^k R(s_k, \pi^\star(s), o)\right) = 
			\min_{o} \E\left(Q^\star(s, \pi^\star(s), o)\right) 
		\end{aligned}
	\end{equation*}
} 
\noindent for $s=s_0$. Note that the optimal policy may not be deterministic, as opposed to reinforcement learning for unconstrained Markov Decision Processes, where there is always an optimal policy that is deterministic. Also, not that we will get different $Q$ tables for different initial states here. Therefore, $Q^\star$ is in fact dependent on and varies with respect to $s_0$. A more proper notation would be to use $Q^\star_{s_0}$, but we omit the indexing with respect to $s_0$ for ease of notation. 
It's relevant to introduce the operator 
\[
\begin{aligned}
(\bT Q)(s,a, o) &= 
R(s, a, o) + \gamma \cdot 
\E\left(Q(s_+, \pi^\star(s_+), o)\right)\\
\end{aligned}
\]
which $\pi^*$ appears in Equation (\ref{pistar0}).
It's not hard to check that the operator $\mathbf{T}$  is not a contraction, so the standard $Q$-learning that is commonly used for reinforcement learning in Markov decision processes with discounted rewards can't be applied here.

In the case we don't know the process $P$ and the reward function $R$, we will not be able to take advantage of the Bellman equation directly. The following results show that we will be able to design an algorithm that always converges to $Q^\star$.

\begin{thm}
	\label{algo}
	Consider a  zero-sum Markov-Bandit game given by the tuple $(S, A, O, P, R)$ where $(S, A, P)$ is unichain, and suppose that $R$ is bounded by some constant and known aprior. 
	Let $Q=Q^\star$ and $\pi^\star$ be solutions to 
	\begin{equation}
		\label{bellman}
		\begin{aligned}
			Q(s,a,o) 	&= R(s, a, o) + \gamma \cdot
			\E\left(Q(s_+, \pi^\star(s_+), o)\right)\\
			(\pi^\star(s), o^\star) &= \arg \max_{\pi\in \Pi} \min_{o} \E\left(Q(s, \pi(s), o)\right) \\
			\pi^\star(s) &= \argmax_{\pi\in \Pi} \E\left(Q(s, \pi(s), o^\star)\right) \\
		\end{aligned}
	\end{equation}		
	Let $\alpha_k(s, a, o) = \alpha_k \cdot 1_{(s,a,o)}(s_k,a_k,o_k)$ satisfy 
	\begin{equation}
		\label{alpha}
		\begin{aligned}
			&0\le \alpha_k(s, a, o) < 1,\ \sum_{k=0}^\infty \alpha_k(s,a,o) = \infty,\\
			& \sum_{k=0}^\infty \alpha_k^2(s,a,o) < \infty,\ \forall (s,a,o) \in S\times A\times O.
		\end{aligned}
	\end{equation}
	Then, the update rule
	\begin{equation}
		\label{q-learning}
		\begin{aligned}
			(\pi_k, o_k)	&= \arg \max_{\pi\in \Pi} \min_{o} \E (Q_k(s_+, \pi(s_+), o))\\
			Q_{k+1} (s,a,o_k)&= \\
			(1-&\alpha_k(s,a,o_k))Q_k(s,a,o_k)+\alpha_k(s,a,o_k)\\
			&\times  (R(s,a,o_k)+\gamma  \E(Q_{k}(s_{+}, \pi_k(s_{+}), o_k)))
		\end{aligned}
	\end{equation}
	converges to $Q^\star$ with probability 1. Furthermore, the optimal policy $\pi^\star \in \Pi$ given by
	(\ref{pistar0}) maximizes
	(\ref{minreward}) with respect to the initial state $s = s_0$. That is, 
	\begin{equation*}
		\begin{aligned}
			(\pi^\star(s_0), o^\star) &= \arg \max_{\pi\in \Pi} \min_{o} \E\left(Q^\star(s_0, \pi(s_0), o)\right) \\
			\pi^\star(s) &= \argmax_{\pi\in \Pi} \E\left(Q^\star(s, \pi(s), o^\star)\right) \\
		\end{aligned}
	\end{equation*}
	
\end{thm}

\subsection{Expected Average Rewards}       
The agent's objective is to maximize the minimal average reward obtained due to the opponent's malicious actions, that is maximizing the total reward given by
\begin{equation}
\label{minreward2}
\min_{o\in O} \lim_{T\rightarrow \infty} \E\left(\frac{1}{T}\sum_{k=0}^{T-1}  R(s_k, a_k, o)\right) 
\end{equation} 
for some initial state $s_0\in S$. Note that this problem is different from the zero-sum game considered in \citep{mannor:2004} where the opponent has to pick a fixed value for its action, $o_k = o$, as opposed to the work in \citep{mannor:2004} where $o_k$ is allowed to vary over time. Thus, from the opponent's point of view, the opponent is performing bandit optimization.

Under Assumption \ref{recurrentstate} and for a given stationary policy $\pi$, the value of
\begin{equation}
\label{vo}
V(o)\triangleq\lim_{T\rightarrow \infty}\E\left(\frac{1}{T}\sum_{k=0}^{T-1}  R(s_k, \pi(s_k), o)\right)
\end{equation} 
is independent of the initial state $s_0$ for any fixed value of the parameter $o$. 
We will make this  standard assumption in Markov decision process control problems.
\begin{prp}
	Consider an MDP $(S, A, P)$ with a total reward (\ref{vo}) for a fixed number $o$.  Under Assumption \ref{recurrentstate} and for a fixed stationary policy $\pi$, there exists a number $v(o)$ and a vector $H(s,o)=(H(S_1, o), ..., H(S_n, o))\in \mathbb{R}^{n}$, such that for each $s\in S$, we have that
	\begin{equation}
	\label{H}
	\begin{aligned}
	&H(s, o)  + v(o)= 
	\E \Big(R(s, \pi(s), o) + \sum_{s_+\in S} P(s_+\mid s, \pi(s))H(s_+,o)\Big).
	\end{aligned}
	\end{equation} 
	Furthermore, the value of (\ref{vo}) is $V(o) = v(o)$. 	
\end{prp}
\begin{proof}
Consult \citep{bertsekas:2005}.
\end{proof}
Introduce
\begin{equation}
	\label{Qreward2}
	\begin{aligned}
	Q(s, a, o) - v(o) = &  R(s, a, o) +  \sum_{s_+\in S} P(s_+\mid s, a)H(s_+, o)
	\end{aligned}
\end{equation} 
and let $Q^\star$, $v^\star$, and $H^\star$ be solutions to Equation (\ref{H})-(\ref{Qreward2}) corresponding to the optimal policy $\pi^\star$ that maximizes (\ref{minreward2}). Then we have that
\begin{equation}
\label{Qrewardopt}
\begin{aligned}
\pi^\star(s) &= \argmax_{\pi\in \Pi}\min_{o\in O} \E\left(Q^\star(s, \pi(s), o)\right)\\
H^\star(s, o) &= \E\left(Q^\star(s, \pi^\star(s), o)\right)\\
Q^\star (s, a, o)- v^\star(o) 
&= R(s, a, o) +  \sum_{s_+\in S} P(s_+\mid s, \pi^\star(s))H^\star(s_+, o) 
\end{aligned}
\end{equation}
In the case we don't know the process $P$ and the reward function $R$, we will not be able to take advantage of (\ref{Qreward2}) directly. 
Introduce the operator

We will make some additional assumptions that will be used in the learning of $Q^\star$ in the average reward case. We start off by introducing a sequence of learning rates $\{\beta_k\}$ and assume that this sequence satisfies the following assumption. 

\begin{assump}[Learning rate]
	\label{lr}
	The sequence $\beta_k$ satisfies:
	\begin{enumerate}
		\item $\beta_{k+1}\le \beta_k$ eventually
		\item For every $0<x<1$, $\sup_k \beta_{\lfloor xk \rfloor}/ \beta_k<\infty$
		\item $\sum_{k=1}^{\infty} \beta_k = \infty$ and $\sum_{k=0}^{\infty} \beta_k^2 < \infty$.
		\item For every $0<x<1$, the fraction
		$$
		\frac{\sum_{k=1}^{\lfloor yt\rfloor} \beta_k}{\sum_{k=1}^t \beta_k} %
		$$
		converges to 1 uniformly in $y\in [x, 1]$ as $t\rightarrow \infty$.
	\end{enumerate}
\end{assump}
For example, $\beta_k = \frac{1}{k}$ and $\beta_k = \frac{1}{k\log k }$ (for $k>1$) satisfy Assumption \ref{lr}. 

Now define $N(k, s, a, o)$ as the number of times that state $s$ and actions $a$ and $o$ were played up to time $k$, that is 
$$
N(k, s, a, o) = \sum_{t=1}^k 1_{(s, a, o)}(s_t, a_t, o_t).
$$

The following assumption is needed to guarantee that all combinations of the triple $(s, a, o)$ are visited often.
\begin{assump}[Often updates]
	\label{ou}
	There exists a deterministic number $d>0$ such that for every $s\in S$, $a\in A$, and $o\in O$, we have that
	$$
	\liminf_{k\rightarrow \infty} \frac{N(k, s, a, o)}{k} \ge d
	$$
	with probability 1.
\end{assump}

\begin{defi}
	\label{phi}
	We define the set $\Phi$ as the set of all functions $f:\R^{n\times m\times q}\rightarrow \R$ such that
	\begin{enumerate}
		\item $f$ is Lipschitz
		\item For any $c\in \R$, $f(cQ) = cf(Q)$
		\item For any $r\in \R$ and $\widehat{Q}(s, a, o) = Q(s, a, o)+r$ for all 
		$(s, a, o)\in \R^{n\times m\times q}$, we have $f(\widehat{Q}) = f(Q) + r$
	\end{enumerate}
\end{defi}
For instance, $f(Q) = \frac{1}{|S| |A| |O|} \sum_{s,a,o} Q(s,a,o)$ belongs the the set $\Phi$.

The next result shows that we will be able to design an algorithm that always converges to $Q^\star$. 
\begin{thm}
	\label{algo2}
	Consider a Markov-Bandit zero-sum game given by the tuple $(S, A, O, P, R)$ %
	and suppose that $R$ is bounded. Suppose that Assumption \ref{recurrentstate}, 
	\ref{lr}, and \ref{ou} hold. 
	Let %
	$f\in \Phi$ be given, where the set $\Phi$ is defined as in Definition \ref{phi}.
	Then, the asynchronous update algorithm given by 
	\begin{equation}
	\label{optQ}
	\begin{aligned}
	&Q_{k+1}(s,a,o) = Q_{k}(s,a,o) + 1_{(s, a, o)}(s_k, a_k, o_k)\times  \\
	&~~~\times \beta_{N(k, s, a, o)} \max_{\pi\in \Pi} \min_{o_k\in O} (R(s, a, o_k)
	+ \E(Q_k(s_{k+1}, \pi(s_{k+1}), o_k))- \\ 
	&\hspace{3.5cm} - f(Q_k)- Q_k(s, a, o))
	\end{aligned}
	\end{equation}
	converges to $Q^\star$ in (\ref{Qrewardopt}) with probability 1. Furthermore, the optimal policy $\pi^\star \in \Pi$ given by (\ref{Qrewardopt}) maximizes (\ref{minreward2}).
\end{thm}

\section{Reinforcement Learning for Constrained Markov Decision Processes}
\label{game-approach}
\subsection{Discounted Rewards}
Consider the optimization problem of finding a stationary policy $\pi$ subject to the initial state $s_0=s$ and the constraints (\ref{dis}), 
that is
\begin{equation}
\label{optproblem}
\begin{aligned}
\textup{find} ~~& {\pi\in \Pi}\\
\textup{s. t.}~~ 		& \E\left( \sum_{k=0}^{\infty}\gamma^k r^j(s_k, \pi(s_k))\right) \ge 0\\ & \textup{for }   j=1, ..., J. %
\end{aligned}
\end{equation}
The next theorem states that the optimization problem  (\ref{optproblem})
is equivalent to a zero-sum Markov-Bandit game, in the sense that an optimal strategy of the agent in the zero-sum game is also optimal for (\ref{optproblem}). %

\begin{thm}
	\label{0sumgame}
	Consider optimization problem (\ref{optproblem}) and suppose it's feasible and that Assumption \ref{r} holds.
	Let $\pi^\star$ be an optimal stationary policy in the zero-sum game
	\begin{equation}
	\label{lagrange}
	v(s_0) =
	\max_{\pi\in\Pi}  \min_{j \in [J]} 
	\E \left(\sum_{k=0}^{\infty}\gamma^k r^j(s_k, \pi(s_k))\right).
	\end{equation}
	Then, $\pi^\star$ is a feasible solution to (\ref{optproblem}) if and only if $v(s_0)\ge 0$.
\end{thm}

The interpretation of the game (\ref{lagrange}) is that the minimizer chooses index $j\in[J]$.

Now that we are equipped with Theorem \ref{algo} and \ref{0sumgame}, we are ready to state and prove our next result.

\begin{thm}
	\label{mainthm}
	Consider the constrained MDP problem (\ref{optproblem}) and suppose that it's feasible and that Assumption \ref{unichain} and \ref{r} hold. Also, introduce $O=[J]$, $o=j$, and
	\[
	R(s,a,o) = R(s,a,j) \triangleq r^j(s,a), ~~~~~  j=1, ..., J.
	\]
	Let $Q_k$ be given by the recursion according to  (\ref{q-learning}).  Then, $Q_k \rightarrow Q^\star$ as $k\rightarrow \infty$ where $Q^\star$ is the solution to (\ref{bellman}).
	Furthermore, the policy
	\begin{equation}
	\label{pistar}
	\pi^\star(s) = \argmax_{\pi\in \Pi} \min_{o\in O} \E\left(Q^\star(s, \pi(s), o)\right)  
	\end{equation}
	is an optimal solution to (\ref{optproblem}) for all $s\in S$.
\end{thm}
\begin{proof}
	According to Theorem \ref{0sumgame}, (\ref{optproblem}) is equivalent to the zero-sum game (\ref{lagrange}), which is equivalent to the zero-sum Markov-Bandit game given by $(S, A, O, P, R)$ with the objective
	\begin{equation}
	\label{cc0}
	\max_{\pi\in \Pi}\min_{o\in O}\E\left( \sum_{k=0}^{\infty}\gamma^k R(s_k, \pi(s_k), o)\right). 
	\end{equation}
	Assumption \ref{r} implies that $|R(s,a, o)|\le 2c$ for all $(s,a,o)\in S\times A\times O$.
	Now let $Q^\star$ be the solution to the maximin Bellman equation (\ref{bellman}). 
	According to Theorem \ref{algo}, $Q_k$ in the recursion given by (\ref{alpha})-(\ref{q-learning}) converges to $Q^\star$ with probability 1. By, definition, the optimal policy $\pi^\star$ achieves the value  of the zero-sum Markov-Bandit game in (\ref{pistar}), and thus achieves the value of (\ref{cc0}).  Hence, 
	\[
	\pi^\star(s) = \argmax_{\pi\in \Pi} \min_{o\in O} \E\left(Q^\star(s, \pi(s), o)\right)  
	\]
	and the proof is complete.
\end{proof}

Finally, the algorithm for Constrained Markov Decision Process with Discounted Reward is shown in Alg. \ref{alg:discounted}.
\begin{algorithm*}[tb]
	\caption{Zero Sum Markov Bandit Algorithm for CMDP with Discounted Reward}
	\label{alg:discounted}
	\begin{algorithmic}[1]
		\STATE Initialize $Q(s,a,o)\leftarrow 0$ for all $(s,a,o)\in\mathcal{S}\times\mathcal{A}\times\mathcal{O}$. Observe $s_0$ and Initial $a_0$ randomly. Select $\alpha_k$ according to Eq. \eqref{alpha}
		\FOR{Iteration $k=0,...,K$} 
		\STATE Take action $a_k$ and observe next state $s_{k+1}$
		\STATE $\pi_{k+1},o_k=\arg\max\limits_{\pi_{k+1}}\min\limits_{o\in\mathcal{O}} Q(s_{k+1}, \pi_{k+1}(s_{k+1}),o_k)$
		\STATE $Q(s_k,a_k,o_k)\leftarrow(1-\alpha_k)Q(s_k,a_k,o_k)+\alpha_k[r(s_k,a_k,o_k)+\gamma\mathbf{E}(Q(s_{k+1},\pi_{k+1}(s_{k+1}),o_k)]$
		\STATE Sample $a_{k+1}$ from the distribution $\pi_{k+1}(\cdot\vert s_{k+1})$
		\ENDFOR  
	\end{algorithmic}
\end{algorithm*}
In line 1, we initialize the Q-table, observe $s_0$ and select $a_0$ randomly. In line 3, we take the current action $a_k$ and observe the next state $s_{k+1}$ so that we can compute the max-min operator in line 4 based on the first line of Eq. \eqref{q-learning}. Line 5 updates the Q-table according to the second line of Eq. \eqref{q-learning}. Line 6 samples the next action from the policy gotten from the line 4.

\subsection{Expected Average Rewards}
Consider the optimization problem of finding a stationary policy $\pi$ subject to the constraints (\ref{ave}), that is
\begin{equation}
\label{optproblem2}
\begin{aligned}
\textup{find} ~~& {\pi\in \Pi}\\
\textup{s. t.}~~ & \lim_{T\rightarrow \infty} \E\left(\frac{1}{T}\sum_{k=0}^{T-1} r^j(s_k, \pi(s_k))\right)  \ge 0\\ &\textup{for }  j=1, ..., J.
\end{aligned}
\end{equation}
The next theorem states that the optimization problem  (\ref{optproblem2})
is equivalent to a zero-sum Markov-Bandit game, in the sense that an optimal strategy of the agent in the zero-sum game is also optimal for (\ref{optproblem2}). %
\begin{thm}
	\label{0sumgame2}
	Consider optimization problem (\ref{optproblem2}) and suppose that Assumption \ref{recurrentstate} and \ref{r} hold. 
	Let $\pi^\star$ be an optimal policy in the zero-sum game
	\begin{equation}
	\label{lagrange2}
	v =
	\max_{\pi\in \Pi}\min_{j \in [J]} ~~ \lim_{T\rightarrow \infty} \E\left(\frac{1}{T}\sum_{k=0}^{T-1}  r^j(s_k, \pi(s_k))\right).
	\end{equation}
	Then, $\pi^\star$ is a solution to (\ref{optproblem2}) if and only if $v\ge 0$.
\end{thm}

Now that we are equipped with Theorem \ref{algo2} and \ref{0sumgame2}, we are ready to state and proof the second main result.

\begin{thm}
	\label{mainthm2}
	Consider the constrained Markov Decision Process problem   (\ref{optproblem2}) and suppose that Assumption \ref{recurrentstate} and \ref{r} hold. Introduce $O=[J]$, $o=j$ and
	\[
	R(s,a,o) = R(s,a,j) \triangleq r^j(s,a), ~~~~~  j=1, ..., J
	\] 
	Let $Q_k$ be given by the recursion according to  (\ref{optQ}) and suppose that Assumptions \ref{lr} and \ref{ou} hold. Then, $Q_k \rightarrow Q^\star$ as $k\rightarrow \infty$ where $Q^\star$ is the solution to (\ref{Qrewardopt}).
	Furthermore, the policy
	\begin{equation}
	\label{pistar2}
	\pi^\star(s) = \argmax_{\pi\in \Pi} \min_{o\in O} \E\left(Q^\star(s, \pi(s), o)\right)  
	\end{equation}
	is a solution to (\ref{optproblem}) for all $s\in S$.
\end{thm}

\begin{proof}
	According to Theorem \ref{0sumgame2}, (\ref{optproblem2}) is equivalent to the zero-sum Markov-Bandit game (\ref{lagrange2}), which is equivalent to the zero-sum Markov-Bandit game given by the tuple $(S, A, O, P, R)$ with the objective
	\begin{equation}
	\label{cc20}
	\max_{\pi\in \Pi}\min_{o\in O} ~~ \lim_{T\rightarrow \infty} \E\left(\frac{1}{T}\sum_{k=0}^{T-1} R(s_k, \pi(s_k), o)\right).
	\end{equation}
	Assumption \ref{r} implies that $|R(s,a, o)|\le 2c$ for all $(s,a,o)\in S\times A\times O$.
	Now let $Q^\star$ be the solution to the maximin optimality equation (\ref{Qrewardopt}). 
	According to Theorem \ref{algo2}, $Q_k$ in the recursion given by (\ref{optQ}) converges to $Q^\star$ with probability 1 under Assumptions \ref{recurrentstate},  \ref{r}, \ref{lr}, and \ref{ou}. By definition, the optimal policy 
	$\pi^\star$ maximizes the expected average reward of the zero-sum Markov-Bandit game (\ref{cc20}). Hence,
	\[
	\pi^\star(s) = \argmax_{\pi\in \Pi} \min_{o\in O} \E\left(Q^\star(s, \pi(s), o)\right)  
	\]
	and the proof is complete.
\end{proof}

The algorithm for Constrained Markov Decision Process with Discounted Reward is in Alg. \ref{alg:average}.
\begin{algorithm*}[tb]
	\caption{Zero Sum Markov Bandit Algorithm for CMDP with Average Reward}
	\label{alg:average}
	\begin{algorithmic}[1]
		\STATE Initialize $Q(s,a,o)\leftarrow 0$ and $N(s,a,o)\leftarrow 0$ for all $(s,a,o)\in\mathcal{S}\times\mathcal{A}\times\mathcal{O}$. Observe $s_0$ and initialize $a_0$ randomly
		\FOR{Iteration $k=0,...,K$} 
		\STATE Take action $a_k$ and observe next state $s_{k+1}$
		\STATE $\pi_{k+1},o_k=\arg\max\limits_{\pi_{k+1}}\min\limits_{o\in\mathcal{O}} \bigg[R(s_k,a_k,o_k)+Q(s_{k+1}, \pi_{k+1}(s_{k+1}),o_k)\bigg]$
		\STATE $t=N(s_k,a_k,o_k)\leftarrow N(s_k,a_k,o_k)+1$; $\alpha_t=\frac{1}{t+1}$
		\STATE $f=\frac{1}{\vert S\vert \vert A\vert O\vert}\sum_{s,a,o}Q(s,a,o)$
		\STATE $y=R(s_k,a_k,o_k)+\mathbf{E}[Q(s_{k+1},\pi_{k+1}(s_{k+1}),o_k]-f$
		\STATE $Q(s_k,a_k,o_k)\leftarrow(1-\alpha_t)Q(s_k,a_k,o_k)+\alpha_k*y$
		\STATE Sample $a_{k+1}$ from the distribution $\pi_{k+1}(\cdot\vert s_{k+1})$
		\ENDFOR  
	\end{algorithmic}
\end{algorithm*}
The most part of this algorithm is similar to Algorithm \ref{alg:discounted}. However, in line 1, we initialize the $N$ table, which records how many times $(s,a,o)$ has been met in the learning process and $N$ table is updated in line 5. Besides, in Line 6, $f$ is computed according to Def. \ref{phi}. Finally, in line 8, Q-table is updated according to the Eq. \eqref{algo2}.

\section{Simulations}
\label{sim}
In this section we will consider three examples for discounted rewards. The first two examples will consider single state. The first example will work out the first few steps of the proposed algorithm. The third example will be multi-state example based on Example \ref{eg:1} described in the Introduction. 

\subsection{Static Process Example 1}

	In this subsection, we consider  an example with  1 state (denoted as $1$), 2 actions (denoted as $1,2$), and two constraints. Let the reward function for the two constraints, $r^j(s,a)$ be given as 
	\begin{equation}
		r^1(1,1)=1 \quad r^1(1,2)=-1\qquad r^2(1,1)=-1 \quad r^2(1,2)=1
	\end{equation}
	The aim of this example is  to find a feasible policy that satisfies the discounted constraints. We let  $\gamma=\frac{1}{2}$ in this example. Since there is only a single state, we will ignore the first variable of state in the following. We note that the only stationary policy that satisfies the constraints in this example is $\pi(1)=\pi(2)=0.5$ due to the symmetry of the two constraints. We will now illustrate that the proposed algorithm will achieve a feasible  policy that satisfies the constraints.

	First, we define the reward function $R(a,o)$ for Markov zero-sum Bandit Game, $a,o\in \{1,2\}$ as
	\begin{equation}
	R(1,1)=1 \quad R(2,1)=-1\qquad R(1,2)=-1 \quad R(2,2)=1
	\end{equation}

	We let the initial value for the Q-function be 0 and assume that the action for $k=0$ is $1$. For the learning rate, we adopt $\alpha_k=\frac{1}{k+1}$. We also label the policy in time-step $i$ as $\pi_i$. 	According to  Theorem 4, we can use the update rule in Eq. \eqref{q-learning}  to obtain the feasible policy. For $k=0$, we have 
	\begin{equation}
	\begin{aligned}
		(\pi_1, o_0)&= \arg \max_{\pi\in \Pi} \min_{o\in O} Q_0(\pi_0(s),o)
	\end{aligned}
	\end{equation} 
	Since $Q_0=0$ for all $(a,o)\in\mathcal{A}\times\mathcal{O}$ and then the objective is not dependent on $\pi$, any arbitrarily policy can be used. Let us choose $\pi$ as a half-half policy such that $\pi_1(1)=\pi_1(2)=0.5$ and assume $a_1=2$. Similarly, $o_0$ can be arbitrary and we assume $o_0=1$. We also let $a_0=1$.  Using $a_0=1, o_0=1, \pi_1(1)=\pi_1(2)=0.5$, the Q-table update is given as
	
	\begin{equation}
		\begin{aligned}
		Q_{1}(1,1) &= (1-\alpha_0(1,1))Q_0(1,1)+\alpha_0(1,1)(R(1,1)+\gamma \E(Q_{0}(\pi_0, 1)))\\
		&=R(1,1)=1
		\end{aligned}
	\end{equation}
	At the end of $k=0$, we get $Q_1(1,1)=1$ and $Q_1(1,2)=Q_1(2,1)=Q_1(2,2)=0$.\\

	For $k=1$, we have
	\begin{equation}
	\begin{aligned}
		(\pi_2, o_1)&= \arg \max_{\pi\in \Pi} \min_{o\in O} Q_1(s,\pi(s),o)
	\end{aligned}
	\end{equation} 
	Since $Q_1(2,1)=Q_1(2,2)=0$, the maxmin problem will again have result 0 whatever the policy $\pi_2$ is. Thus, we still assume that $\pi_2(1)=\pi_2(2)=0.5$ and next action $a_2=1$. However, it follows that $o_1=2$ because $Q_1(1,1)=1$. Since $a_1=2, o_1=2, \pi_2(1)=\pi_2(2)=0.5$, the Q-table update is
	\begin{equation}
	\begin{aligned}
	Q_2(2,2) &= (1-\alpha_1(2,2))Q_1(2,2)+\alpha_1(2,2)(R(2,2)+\gamma  \E(Q_{1}(\pi_1, 2)))\\
	&=0.5* 0 + 0.5 * (1+0.5* 0)=0.5
	\end{aligned}
	\end{equation}
	At the end of $k=1$, we get $Q_2(1,1)=1$, $Q_2(2,2)=0.5$ and $Q_2(1,2)=Q_2(2,1)=0$.\\

	For $k=2$, we have
	\begin{equation}
	\begin{aligned}
		(\pi_3, o_2)&= \arg \max_{\pi\in \Pi} \min_{o\in O} Q_2(s,\pi(s),o)
	\end{aligned}
	\end{equation}
	To solve this problem, it is equivalent to solve the following problem
	\begin{equation}\label{equivalent}
		\begin{aligned}
		\arg\max_{z}&\quad z\\
		s.t. &\quad z\leq Q_2(s,\pi(s),o)\quad \text{for}\quad o=1,2
		\end{aligned}
	\end{equation}
	Assume $\pi_3(1)=p,\pi_3(2)=1-p$, this is equivalent to solve the equation that $p*Q_2(1,1)+(1-p)*Q_2(1,2)=p*Q_2(1,2)+(1-p)*Q_2(2,2)$, which gives the result $\pi_3(1)=\frac{1}{3}$ and $\pi_3(2)=\frac{2}{3}$ and we assume the next action $a_3=2$. Due to the equality in the above equation, $o_2$ can again can be arbitrary and we assume $o_2=2$. Since $a_2=1$, the Q-table update is
	\begin{equation}
	\begin{aligned}
		Q_3(1,2) &= (1-\alpha_2(1,2))Q_2(1,2)+\alpha_2(1,2)(R(1,2)+\gamma  \E(Q_{2}(\pi_2, 2)))\\
		&=\frac{2}{3}* 0 + \frac{1}{3} * [-1 + 0.5 * (\frac{1}{3}* Q_2(1,2)+ \frac{2}{3}*Q_2(2,2))]=-\frac{5}{18}
	\end{aligned}
	\end{equation}
	At the end of $k=2$, we get $Q_3(1,1)=1$, $Q_3(2,2)=0.5$ and $Q_3(1,2)=-\frac{5}{18}$ and $Q_3(2,2)=0$.\\

	For $k=3$, we have
	\begin{equation}
	\begin{aligned}
	(\pi_4, o_3)&= \arg \max_{\pi\in \Pi} \min_{o\in O} Q_3(s,\pi(s),o)	
	\end{aligned}
	\end{equation} 
	We need  to solve the problem in the Equation \eqref{equivalent} to get the result of $\pi_4$ and the result is $\pi_4(1)=\frac{7}{16}$ and $\pi_4(2)=\frac{9}{16}$ and $o_3$ can be arbitrary, thus we assume that $o_3=1$. Since $a_3=2$, the Q-table update is given as
	\begin{equation}
		\begin{aligned}
		Q_4(2,1) &= (1-\alpha_3(2,1))Q_3(2,1)+\alpha_3(2,1)(R(2,1)+\gamma  \E(Q_{3}( 1, 1)))\\
		&=\frac{3}{4}* 0 + \frac{1}{4} * (-1+ 0.5* (\frac{7}{16} * Q_3(1,1)+\frac{9}{16}*Q_3(2,1)))=-\frac{3}{8}-\frac{1}{4}=-\frac{25}{128}
		\end{aligned}
	\end{equation}
	At the end of $k=3$, we get $Q_4(1,1)=1$, $Q_4(2,2)=0.5$ and $Q_4(1,2)=-\frac{5}{18}$ and $Q_4(2,1)=-\frac{25}{128}$.\\
	
	Based on these steps, we can keep on computing the update for Q-table. However, the computation is hard to do manually, and involves random choice of actions based on policy $\pi$. Thus, we simulate the performance of the algorithm and the Q-values  $Q_k(i,j)$ for iterations $k$ are depicted in Fig. \ref{Fig:cov}.
	
	\begin{figure}[htb]
		\centering
		\resizebox{.7\textwidth}{!}{  
			\input{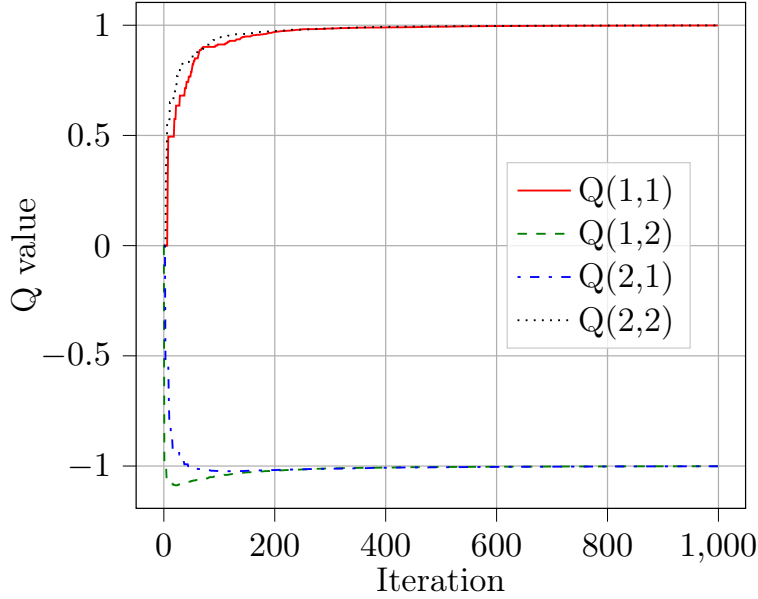}
		}
		\caption{Convergence result for  Example 1 }
		\label{Fig:cov}	
	\end{figure}
	
	We note that $Q_k(1,1)$ and $Q_k(2,2)$ converges to $1$, while $Q_k(1,2)$ and $Q_k(2,1)$ converges to $-1$. According to the optimal Bellman equation,
	\begin{equation}
	Q^*(s,a,o)=R(s, a, o) + \gamma \cdot 
	\E\left(Q^*(s_+, \pi^\star(s_+), o)\right)
	\end{equation}
	we know $Q^*(1,1)=1+0.5*[0.5*Q^*(1,1)+0.5*Q^*(2,1)]$, which means
	\begin{equation}
	3Q^*(1,1)=4+Q^*(2,1)
	\end{equation}
	Similarly, we have
	\begin{equation}
	3Q^*(2,1)=-4+Q^*(1,1)
	\end{equation}
	Combining these two equations, we have $Q^*(1,1)=-Q^*(2,1)=-1$. Similarly, $Q^*(2,2)=-Q^*(1,2)=-1$. Thus, we see that the algorithm successfully have the whole $Q$ table converges to $Q^*$, which shows the correctness of the theorem. Moreover,
	\begin{equation}
	\pi^*=\argmax_{\pi\in \Pi}\min_{o\in O}Q^*(s,\pi(s),o)
	\end{equation}
	which gives $\pi^*(1\vert s)=\pi^*(2\vert s)=0.5$ and we know this is the only feasible policy. Thus, we see that the Q-values of the proposed algorithm converges to that of the optimal policy and the policy converges to the only feasible policy in this example.

\subsection{Static Process Example 2}

We consider a  static process (that is, the state is constant) and an agent that takes action from the action set $A=\{1, 2, 3\}$. There are three objectives given by the reward functions $r_1, r_2$, and $r_3$ defined as 
\begin{equation*}
r^j(a) = 
\begin{cases}
\frac{1}{2} & \text{if $a = j$} \\
0 & \text{otherwise}
\end{cases}
\end{equation*}
Note that we have dropped the dependence of the reward functions $r_j$ on the state $s$ as the state $s$ is assumed to be constant. Let the discount factor be $\gamma =\frac{1}{2}$ and let $$\alpha_0 = \alpha_1 = \alpha_2 = \alpha = \frac{1}{3}.$$
The agent would then be looking for a probability distribution over the set $A$, $\Prob(a)$ for $a\in A$, that simultaneously satisfies the objectives
\begin{equation*}
\E\left( \sum_{k=0}^{\infty}\gamma^k r^j(a_k)\right) \ge \frac{1}{3}, ~~~~~  j=1, 2, 3.
\end{equation*}

\begin{figure}[ht]
	\begin{center}
		\includegraphics[width=8cm]{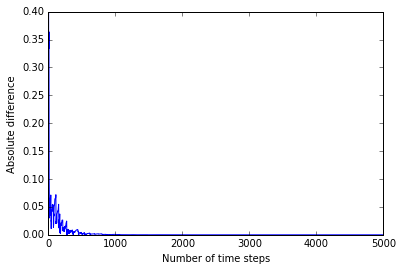}
		\centering
		\caption{A plot of the maximum of 
			$
			|p_1-\hat{p}_1| + |p_2-\hat{p}_3| + |p_3 - \hat{p}_3|
			$ over 1000 iterations, as a function of the number of time steps.}
		\label{diff}
	\end{center}
\end{figure}
Now suppose that the agent takes action $a_k = 1$ with probability $p_1$. Then we have that
$$
\E\left( \sum_{k=0}^{\infty}\gamma^k r^1(a_k)\right) = p_1.
$$
Similarly, we find that if the agent takes the action $a_k = j$ with probability $p_j$, $j=2,3$, then
$$
\E\left( \sum_{k=0}^{\infty}\gamma^k r^j(a_k)\right) = p_j .
$$

Without loss of generality, suppose that $p_1\le p_2 \le p_3$. Now the equality $p_1+p_2+p_3 = 1$ together with the Arithmetic-Geometric Mean Inequality imply that
$$
\frac{1}{3} = \frac{p_1+p_2+p_3}{3}\ge \sqrt[3]{p_1p_2p_3} \ge p_1
$$
with equality if and only if $p_1=p_2=p_3 = \frac{1}{3}$. Thus, in order to satisfy all of the three objectives, the agent's mixed strategy is unique and given by $p_1=p_2 =p_3= \frac{1}{3}$.

We have run 1000 iterations of a simulation of the learning algorithm as given by Theorem \ref{mainthm} over 5000 time steps (with respect to the time index $k$). As the above calculations showed, the probability distribution of the optimal policy is given by $p_1=p_2=p_3 = \frac{1}{3}$. Let $\hat{p}_1, \hat{p}_2, \hat{p}_3$ be the estimated probabilities based on the $Q$-learning algorithm given by Theorem \ref{mainthm}. In Figure~\ref{diff}, we see a plot of the maximum of the total error
$$
|p_1-\hat{p}_1| + |p_2-\hat{p}_2| + |p_3 - \hat{p}_3|
$$ 
over all iterations, 
as a function of the number of time steps. We see that it converges after 1000 time steps and stays stable for the rest of the simulation.

\subsection{Discrete time Single-Server Queue}
	In this subsection, we evaluate the proposed algorithm on a queuing system with a single server in discrete time, which is based on Example \ref{eg:1} described in the Introduction. In this model, we assume there is a buffer of finite size $L$. A possible arrival is assumed to occur at the beginning of the time slot. The state of the system is the number of customers waiting in the queue at the beginning of time slot such that $\vert S\vert=L+1$. We assume there are two kinds of actions, service action and flow action. The service action space is a finite subset $A$ of $[a_{min},a_{max}]$ and $0<a_{min}\leq a_{max}<1$. With a service action $a$, we assume that a service of a customer is successfully completed with probability $a$. If the service succeeds, the length of the queue will reduce by one, otherwise there is no change of the queue. The flow  is a finite subset $B$ of $[b_{min}, b_{max}]$ and $0\leq b_{min}\leq b_{max}<1$. Given a flow action $b$, a customer arrives during the time slot with probability $b$. Let the state at time $t$ be $x_t$. We assume that no customer arrives when state $x_t=L$ and thus can model this by the state update not increasing on customer arrival when $x_t=L$. Finally, the overall action space is the product of service action space and flow action space, i.e., $A\times B$. Given an action pair $(a,b)$ and current state $x_t$, the transition of this system $P(x_{t+1}|x_t,a_t=a,b_t=b)$ is shown in Table \ref{transition}.
	\begin{table*}   
		\caption{Transition probability of the queue system}  
		\label{transition}
		\begin{center}  
			\begin{tabular}{|c|c|c|c|}  
				\hline  
				Current State & $P(x_{t+1}=x_t-1)$ & $P(x_{t+1}=x_t)$ & $P(x_{t+1}=x_t+1)$ \\ \hline
				$1\leq x_t\leq L-1$ & $a(1-b)$ & $ab+(1-a)(1-b)$ & $(1-a)b$ \\ \hline
				$x_t=L$ & $a$ & $1-a$ & $0$ \\ \hline
				$x_t=0$ & $0$ & $1-b(1-a)$ & $b(1-a)$ \\ 
				\hline  
			\end{tabular}  
		\end{center}  
	\end{table*}
	
	Given this transition probability matrix, it is clear that the next state is only decided by the current state and current action, which means it is a Markov Decision Process. Moreover, the cost function $c(s,a,b)$ is assumed to be only related to the length of the queue and is a increasing linear function with respect to the state. It is reasonable because the cost can be seen as the expected waiting time by the Little's law. The less time customers wait, the lower the cost. Besides, there are two constraint functions, related to the service action and flow action, respectively. The service constraint function $c^1(s,a,b)$ is assumed to be only related to $a$ and increasing with the service action $a$, while the flow constraint function $c^2(s,a,b)$ is assumed to be only related to $b$ and decreasing with the service action $b$.
	
	Assuming that $\gamma=0.5$, we want to optimize the total discounted reward collected and satisfies two constraints with respect to service and flow simultaneously. Thus, the overall optimization problem is given as 
	\begin{equation}\label{opt_pro}
	\begin{aligned}
	\min\limits_{\pi^a, \pi^b}&\quad\mathbb{E}\bigg[\sum_{t=1}^{\infty}\gamma^tc(s_t, \pi^a(s_t), \pi^b(s_t))\bigg]\\
	s.t.
	&\quad \mathbb{E}\bigg[\sum_{t=1}^{\infty}\gamma^tc^1(s_t, \pi^a(s_t), \pi^b(s_t))\bigg]\leq 0,\quad \mathbb{E}\bigg[\sum_{t=1}^{\infty}\gamma^tc^2(s_t, \pi^a(s_t), \pi^b(s_t))\bigg]\leq 0,
	\end{aligned}
	\end{equation} 
	where $\pi^a_h$ and $\pi^b_h$ are the policies for the service and flow at time slot $h$, respectively. We note that the expectation in the above is with respect to both the stochastic policies and the transition probability. In order to match the constraints satisfaction problem modeled in this paper, we use the bisection algorithm on $\delta$ and transform the above problem to the following problem.
	\begin{equation}\label{trans_pro}
	\begin{aligned}
	Find &\quad \pi=(\pi^a,\pi^b)\\
	s.t.
	&\quad \mathbb{E}\bigg[\sum_{t=1}^{\infty}\gamma^tc(s_t, \pi^a(s_t), \pi^b(s_t))\bigg]\leq \delta\\
	&\quad \mathbb{E}\bigg[\sum_{t=1}^{\infty}\gamma^tc^1(s_t, \pi^a(s_t), \pi^b(s_t))\bigg]\leq 0\\
	&\quad \mathbb{E}\bigg[\sum_{t=1}^{\infty}\gamma^tc^2(s_t, \pi^a(s_t), \pi^b(s_t))\bigg]\leq 0,
	\end{aligned}
	\end{equation} 
	In the setting of the simulation, we choose the length of the queue $L=5$. We let the service action space be $A=[0.3,0.4,0.5,0.6,0.7]$ and the flow action space be $B=[0,0.2,0.4,0.6]$ for all states besides the state $s=L$. Moreover, the cost function is set to be 
	\begin{equation}
		c(s,a,b)=s-5
	\end{equation}
	the constraint function for the service is defined as 
	\begin{equation}
		c^1(s,a,b)=10a-5
	\end{equation} 
	and the constraint function for the flow is 
	\begin{equation}
		c^2(s,a,b)=5(1-b)^2-2
	\end{equation}
	
	For different values of $\delta$, the numerical results are given in Fig. \ref{fig:learn}. To show the performance of the algorithm, we choose the values of $\delta$ close to the real optimum value and thus the figure shows the performance with $\delta=9.5$, $9.55$, $9.575$,  $9.6$, $9.625$, and $9.7$. For each value of $\delta$,  we run the algorithm for $10^5$ iterations. Rather than evaluating the policy in each iteration, we evaluate the policy every 100 iterations, while evaluate at every iteration for the last 100 iterations. In order to get the expected value of the constraints, we collect 10000 trajectories and calculate the average constraint function value among them.  These constraint function values for the three constraints are plotted in Fig. \ref{fig:learn}. For $\delta=9.5$, we see that  the algorithm converges after about 60000 iterations and all three constraints are  larger than 0, which means that we find a feasible policy for the setting $\delta=9.5$.  Moreover, it is reasonable that all three constraints  converge to a same value since the proposed Algorithm \ref{alg:discounted} optimize the minimal value function among $V(s,a,o)$ with respect to $o$. We see that the three constraints for $\delta=9.5$ are close to each other and non-negative, thus demonstrating the constraints are satisfied and $\delta=9.5$ is feasible. 
	
	On the other extreme, we see the case when $\delta=9.7$. We note  that all three constraints are below 0, which means that there is no feasible policy in  this setting. Thus, seeing the cases for $\delta=9.5$ and $9.7$, we note that the optimal objective is between the two values. Looking at the case where $\delta=9.625$, we also note that the service constraint is clearly below zero and the constraints are not satisfied. Similarly, for $\delta=9.55$, the constraints are non-negative - the closest to zero are the service constraints which  are crossing zero every few iterations and thus the gap is within the margin. This shows that the optimal objective is within $9.55$ and $9.625$. However, the judgement is not as evident between the two regimes and it cannot be clearly mentioned from $\delta=9.575$ and $\delta=9.6$ if they are feasible or not since they are not consistently lower than zero after  80,000 iterations like in the case of $\delta=9.625$ and $\delta=9.7$, are are not mostly above zero as for $\delta=9.5$. Thus, looking at the figures, we estimate the value of optimal objective between $9.55$ and $9.625$.

In order to compare the result with the theoretical optimal total reward, we can assume the  dynamics of the MDP is known in advance and use the Linear Programming algorithm to solve the original problem. The result solved by the LP is $9.62$. 

We note that  $Q(s,a,o)$ has  $6\times5\times4\times 3=360$ elements and it is possible that $10^5$ iterations are not  enough to make all the elements in $Q$ table to converge.  Further, sampling $10^4$ trajectories can only achieve an accuracy of $0.1$ with $99\%$ confidence for the constraint function value and we are within that range. Thus, more iterations and more samples (especially more samples) would help improve the achievable estimate from $9.55$ in the algorithm performance.  Overall, considering the limited iterations and sampling in the simulations, we  conclude that the result by the proposed algorithm is close to the optimal result obtained by the Linear Programming.
	\begin{figure*}[htbp]
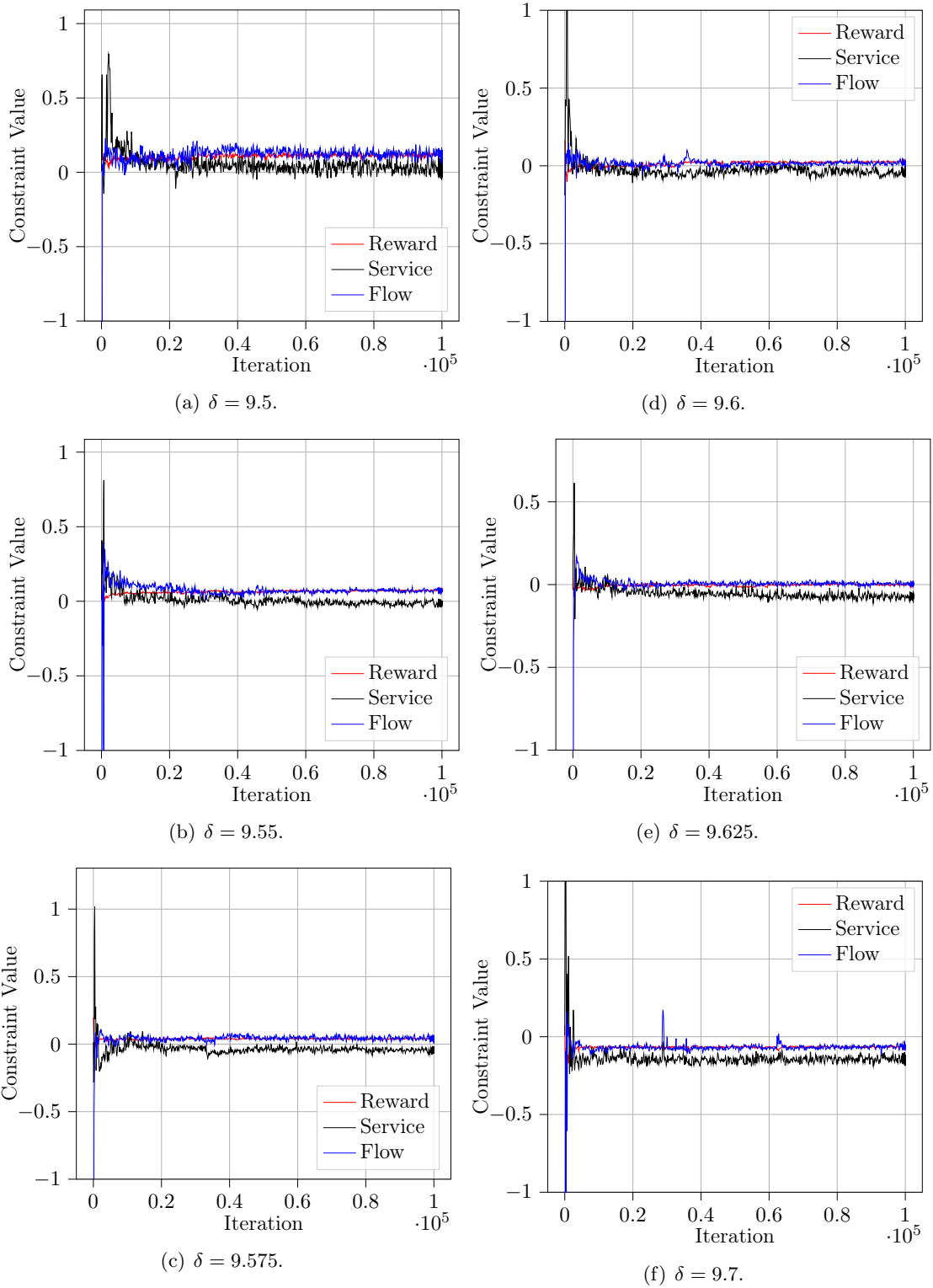

		\centering
			\begin{minipage}[h]{0.48\linewidth}
				\centering
				\subfigure[$\delta=9.5$. ]{
						\resizebox{\textwidth}{!}{%
				\input{bi_section_9.5.tex}}}
				\subfigure[$\delta=9.55$. ]{	\resizebox{\textwidth}{!}{%
					\input{bi_section_9.55.tex}}}
			\subfigure[$\delta=9.575$. ]{\resizebox{\textwidth}{!}{%
					\input{bi_section_9.575.tex}}}
			\end{minipage}%
		\vspace{.1in}
			\begin{minipage}[h]{0.48\linewidth}
				\centering
			\subfigure[$\delta=9.6$. ]{\resizebox{\textwidth}{!}{%
				\input{bi_section_9.6.tex}}}
			\subfigure[$\delta=9.625$. ]{	\resizebox{\textwidth}{!}{%
					\input{bi_section_9.625.tex}}}
				\subfigure[$\delta=9.7$. ]{\resizebox{\textwidth}{!}{%
						\input{bi_section_9.7.tex}}}
			\end{minipage}%
		\centering
		\caption{Value of constraints with iterations for the Zero-Sum Markov Bandit algorithm applied to Discrete time single-server queue. }
		\label{fig:learn}	
	\end{figure*}

\section{Conclusions}
\label{conc}
We considered the problem of optimization and learning for constrained and multi-objective Markov decision processes, for both discounted rewards and expected average rewards. We formulated the problems as zero-sum games where one player (the agent) solves a Markov decision problem and  its opponent solves a bandit optimization problem, which we call Markov-Bandit games. We extended $Q$-learning to solve Markov-Bandit games and proved that our new $Q$-learning algorithms converge to the optimal solutions of the zero-sum Markov-Bandit games, and hence converge to the optimal solutions of the constrained and multi-objective Markov decision problems. We provided  numerical examples and the  simulation results illustrate that the proposed algorithm converges to the optimal policy. 

It would be interesting to combine our algorithms with deep reinforcement learning and study the performance of the maximin $Q$-learning approach in this paper when the value function $Q$ is modeled as a deep neural network.

\bibliography{mybib}

\newpage

\appendix

\section{Proof of Theorem \ref{algo}}
The proof will rely on the following result.
\begin{prp}
	\label{delta}
	The random process $\{\Delta_k\}$ taking values in $\R$ and defined as
	$$
	\Delta_{k+1}(x) = (1-\alpha_k(x))\Delta_k(x) + \alpha_k(x)F_k(x)
	$$
	converges to zero with probability 1 under the following assumptions:
	\begin{itemize}
		\item[i.] For all $x$, $0\le \alpha_k(x)\le 1$, $\sum_k \alpha_k(x) = \infty$, and $\sum_k \alpha_k^2(x)<\infty$
		\item[ii.] $\|\E(F_k(x)\mid \mathcal{F}_k)\|_\infty\le \gamma \|\Delta_k\|_\infty$, with $\gamma<1$ 
		\item[iii.] $\E(F_k-\E(F_k(x))\mid \mathcal{F}_k))^2\le C(1+\|\Delta_k\|_\infty^2)$, for some constant $C>0$ 
	\end{itemize}
	where $\mathcal{F}_k$  is the sigma algebra $\sigma(\Delta_t, F_{t-1}, \alpha_{t-1},  t\le k)$.
\end{prp}
\begin{proof}
	Consult \citep{jaakkola:1994}.	
\end{proof}	
Now let
$$
\Delta_k(s, a, o) = 
Q_k(s, a, o) - Q^\star(s, a, o)
$$

Subtracting 
$Q^\star$ from the right and left hand sides of the second equality in (\ref{q-learning}) implies that
\begin{equation*}
	\begin{aligned}
		&\Delta_{k+1}(s, a, o) = (1-\alpha(s,a,o))\Delta_k(s,a,o)+\\ &~+\alpha(s,a,o)(R(s,a,o)+\gamma  \E(Q_{k}(s_{+}, \pi_k(s_{+}), o)) - Q^\star(s,a,o)).
	\end{aligned}
\end{equation*}
We will show that $\Delta_k$ satisfies the conditions of Proposition \ref{delta}. Introduce the sigma algebra $\mathcal{F}_k=\sigma(\Delta_t, F_{t-1}, \alpha_{t-1},  t\le k)$.

Define
$$
F_k(s,a,o) = 1_{(s,a,o)}(s_k, a_k, o_k)\times (R(s,a,o) + \gamma  \E(Q_{k}(s_{+}, \pi_k(s_{+}), o)) -Q^\star(s,a,o))
$$
If $(s,a,o) \ne (s_k, a_k, o_k)$, then $F_k(s,a,o) = 0$. Else, 
\begin{equation*}
	\begin{aligned}
		\E(F_k(s,a,o)\mid \mathcal{F}_k) &= \sum_{s_{+}} P(s,a,s_{+})1_{(s,a,o)}(s_k, a_k, o_k)\times\left(R(s,a,o) + \right.\\
		&\hspace{0.7cm}\gamma \E(Q_{k}(s_{+}, \pi_k(s_{+}), o)) -Q^\star(s,a,o)\big)\\
		&= \sum_{s_{+}} P(s,a,s_{+})\left(R(s, a, o_k) + \right.\\
		&\hspace{0.7cm}\gamma \E(Q_k(s_{+}, \pi_k(s_+), o_k)) -Q^\star(s,a,o_k)\big)\\
		&= \sum_{s_{+}} P(s,a,s_{+})\left(R(s, a, o_k) +\gamma \E(Q_k(s_{+}, \pi_k(s_+), o_k)) - \right.\\
		&\hspace{0.7cm}  R(s, a, o_k) - \gamma \E(Q^\star(s_+,\pi^\star(s_+),o_k))\big)\\
		&= \gamma\sum_{s_{+}} P(s,a,s_{+})\left( \E(Q_k(s_{+}, \pi_k(s_+), o_k)) - \right.\\
		& \hspace{0.7cm}  \left. \E(Q^\star(s_+,\pi^\star(s_+),o_k))\right)\\
	\end{aligned}
\end{equation*}
If $\E(Q_k(s_{+}, \pi_k(s_+), o_k)) \ge  \E(Q^\star(s_+,\pi^\star(s_+), o_k))$, then
\begin{eqnarray*}
	&&\bigg|\mathbf{E}[Q_k(s_{+}, \pi_k(s_+),o_k)] -
	\mathbf{E}[Q^\star(s_+,\pi^\star(s_+),o_k)]\bigg| \\ 
	&=& \E(Q_k(s_{+}, \pi_k(s_+), o_k)) - 
	\E(Q^\star(s_+,\pi^\star(s_+),o_k)) \\
	&=& R(s, a, o_k) + \E(Q_k(s_{+}, \pi_k(s_+), o_k)) - R(s, a, o_k) -
	\E(Q^\star(s_+,\pi^\star(s_+),o_k)) \\
	&\le& R(s, a, o_k) + \E(Q_k(s_{+}, \pi_k(s_+), o_k)) - R(s, a, o^\star) -
	\E(Q^\star(s_+,\pi^\star(s_+), o^\star)) \\
	&\le&  R(s, a, o^\star) +\E(Q_k(s_{+}, \pi_k(s_+), o^\star)) - R(s, a, o^\star) -
	\E(Q^\star(s_+,\pi^\star(s_+), o^\star)) \\
	&\le& R(s, a, o^\star) + \E(Q_k(s_{+}, \pi_k(s_+), o^\star)) - R(s, a, o^\star) -
	\E(Q^\star(s_+,\pi_k(s_+), o^\star)) \\
	&=& \left|\E(Q_k(s_{+}, \pi_k(s_+), o^\star)
	- Q^\star(s_+,\pi_k(s_+),o^\star))\right| \\
	&\le& \max_{s, a, o} \left|Q_k(s, a, o) -
	Q^\star(s, a, o)\right| \\
	&=&  \left\| Q_k - Q^\star  \right\|_\infty. \label{diffQ}
\end{eqnarray*}
Else, if $\E(Q_k(s_{+}, \pi_k(s_+), o_k)) \le  \E(Q^\star(s_+,\pi^\star(s_+),o_k))$, then
\begin{eqnarray*}
	&&\bigg|\E[Q_k(s_{+}, \pi_k(s_+), o_k)] -
	\E[Q^\star(s_+,\pi^\star(s_+),o_k)]\bigg| \label{diffQ02} \\ 
	&=& \E(Q^\star(s_+,\pi^\star(s_+),o_k)) - \E(Q_k(s_{+}, \pi_k(s_+), o_k))\\
	&\le& \E(Q^\star(s_+,\pi^\star(s_+),o_k)) - \E(Q_k(s_{+}, \pi^\star(s_+), o_k))\\
	&=& \left| \E(Q_k(s_+,\pi^\star(s_+),o_k) - Q^\star(s_{+}, \pi^\star(s_+), o_k)) \right|\\
	&\le& \max_{s, a, o} \left|Q_k(s, a, o) -
	Q^\star(s, a, o)\right| \\
	&=&  \left\| Q_k - Q^\star \right\|_\infty. \label{diffQ2}
\end{eqnarray*}
Thus, 
\begin{equation}
	\label{TQ}
	\begin{aligned}
		&\|\E(F_k(s,a,o)\|_\infty = \\
		&= \gamma \max_{s,a,o}\left| \sum_{s_+} P(s,a,s_{+})\left( \E(Q_k(s_{+}, \pi_k(s_+), \phi_k(s_+))) -  \E(Q^\star(s_+,\pi^\star(s_+),\phi_k(s_+)))\right)\right|\\
		&\le \gamma \max_{s,a,o} \sum_{s_+} P(s,a,s_{+}) \left|\left( \E(Q_k(s_{+}, \pi_k(s_+), \phi_k(s_+))) -  \E(Q^\star(s_+,\pi^\star(s_+),\phi_k(s_+)))\right)\right|\\
		&\le \gamma \max_{s,a,o} \sum_{s_+} P(s,a,s_{+}) \left\| Q_k - Q^\star \right\|_\infty \\
		&= \gamma \left\| Q_k - Q^\star \right\|_\infty \\
		&= \gamma \|\Delta_k\|_\infty
	\end{aligned}
\end{equation}	
where the first inequality follows from the triangle inequality and the fact that $P(s,a,s_+)\ge 0$. 	
Also, we have that
\begin{equation*}
	\begin{aligned}
		&\E(F_k-\E(F_k) \mid \mathcal{F}_k))^2=\\
		&=\gamma^2\E\Big(Q_k(s_{+}, \pi_k(s_+), \phi_k(s_+)) - Q^\star(s_+,\pi^\star(s_+),\phi_k(s_+))-\\ 
		&\hspace{2.25cm}-\sum_{s_+} P(s,a,s_+)\left(Q_k(s_{+}, \pi_k(s_+), \phi_k(s_+)) - Q^\star(s_+,\pi^\star(s_+),\phi_k(s_+))\right)\Big)^2\\
		&=\gamma^2\E\Big(\Delta_k(s_{+}, \pi_k(s_+), \phi_k(s_+)) -\\ 
		&\hspace{2.25cm}-\sum_{s_+} P(s,a,s_+)\left(\Delta_k(s_{+}, \pi_k(s_+), \phi_k(s_+)) \right)\Big)^2\\
		&\le C(1+\|\Delta_k\|_\infty^2). 
	\end{aligned}
\end{equation*}
Thus, $\Delta_k=Q_{k} -Q^\star$ satisfies the conditions of Proposition \ref{delta} and hence converges to zero with probability 1, i. e. $Q_k$ converges to $Q^\star$ with probability 1.

\section{Proof of Theorem \ref{algo2}}	

\begin{lem}
	\label{contraction}
	Let the operator $\bT$ be given by
	{\small
		\begin{equation}
		\label{T1}
		\begin{aligned}
		(\mathbf{T}Q)(s,a,o) %
		&= \sum_{s_+} P(s,a,s_+)\max_{\pi\in \Pi} \min_{o    \in O}  \left(R(s,a,o) + \E(Q(s_+, \pi(s_+), o))\right). 
		\end{aligned}
		\end{equation}
	}	
	Then, 
	$$
	\|\bT Q_1 - \bT Q_2\|_\infty \le \|Q_1 - Q_2\|_\infty.
	$$
\end{lem}
\begin{proof}
	\begin{equation}
	\label{TQ}
	\begin{aligned}
	&\|\bT Q_1 - \bT Q_2\|_\infty =\\
	&= \max_{s,a,o}\left| \sum_{s_+} P(s,a,s_+)\left(\max_{\pi\in \Pi} \min_{o    \in O} (R(s,a,o) +  \E(Q_1(s_+, \pi(s_+), o)))-\right. \right. \\
	& \hspace{2cm} - \left. \max_{\pi\in \Pi} \min_{o    \in O} (R(s,a,o) +  \E(Q_2(s_+, \pi(s_+), o)))\right)\Bigg| \\
	&\le \max_{s,a,o} \sum_{s_+} P(s,a,s_+) \left| \max_{\pi\in \Pi} \min_{o    \in O} (R(s,a,o) +  \E(Q_1(s_+, \pi(s_+), o))) - \right. \\
	& \hspace{2cm} \left. - \max_{\pi\in \Pi} \min_{o    \in O} (R(s,a,o) +  \E(Q_2(s_+, \pi(s_+)), o))\right)\Big| \\
	\end{aligned}
	\end{equation}	
	where the last inequality follows from the triangle inequality and the fact that $P(s,a,s_+)\ge 0$. 
	Without loss of generality, assume that 
	\[
	\begin{aligned}
		&\max_{\pi\in \Pi} \min_{o    \in O} (R(s,a,o) +  \E(Q_1(s_+, \pi(s_+), o))) \\ 
		& \ge \max_{\pi\in \Pi} \min_{o    \in O} (R(s,a,o) +  \E(Q_2(s_+, \pi(s_+), o)))\big).
	\end{aligned} 
	\]
	Introduce
	$$
	(\pi_i, o_i) = \arg \max_{\pi\in \Pi}\min_{o \in O} R(s,a,o) + \E(Q_i(s_+, \pi(s_+), o)).
	$$
	Then, 
	\begin{eqnarray}
	&&\left| \max_{\pi\in \Pi} \min_{o    \in O} \big(R(s,a,o) +  \E(Q_1(s_+, \pi(s_+), o))\big) - \right. \\
	&& \hspace{2cm} \left. - \max_{\pi\in \Pi} \min_{o    \in O} \big(R(s,a,o) +  \E(Q_2(s_+, \pi(s_+), o))\big)\right|  \nonumber\\
	&=&\max_{\pi\in \Pi} \min_{o    \in O} \big(R(s,a,o) +  \E(Q_1(s_+, \pi(s_+), o))\big)- \\
	&& \hspace{2cm} - \max_{\pi\in \Pi} \min_{o\in O} \big(R(s,a,o) +  \E(Q_2(s_+, \pi(s_+), o))\big)  \nonumber\\
	&=&R(s,a,o_1) +  \E(Q_1(s_+, \pi_1(s_+), o_1)) - \left(R(s,a,o_2) +  \E(Q_2(s_+, \pi_2(s_+), o_2))\right)  \nonumber\\
	&\le& R(s,a,o_2) +  \E(Q_1(s_+, \pi_1(s_+), o_2))- \left(R(s,a,o_2) +  \E(Q_2(s_+, \pi_2(s_+), o_2))\right)   \nonumber\\
	&\le& R(s,a,o_2) +  \E(Q_1(s_+, \pi_1(s_+), o_2)) - \left(R(s,a,o_2) +  \E(Q_2(s_+, \pi_1(s_+), o_2))\right)  \nonumber\\
	&=& \left| \E(Q_1(s_+, \pi_1(s_+), o_2))  - Q_2(s_+, \pi_1(s_+), o_2))\right|  \nonumber\\
	&\le&\max_{s_+,a,o} \left| Q_1(s_+, a, o) - Q_2(s_+, a, o) \right|\\  
	&=&  \left\| Q_1 - Q_2 \right\|_\infty. \label{diffQ}
	\end{eqnarray}
	Combining (\ref{TQ})--(\ref{diffQ}) implies that
	\begin{equation}
	\label{contract}
	\begin{aligned}
	&\|\bT Q_1 - \bT Q_2\|_\infty \le\\
	&\le \max_{s,a,o} \sum_{s_+} P(s,a,s_+) \left\| Q_1 - Q_2 \right\|_\infty \\
	&= \left\| Q_1 - Q_2 \right\|_\infty
	\end{aligned}
	\end{equation}	
	and the proof is complete.		
\end{proof}

\begin{lem}
	The operator $\bT$ given by (\ref{T1}) is a span semi-norm, that is
	\begin{equation}
	\label{spannorm}
	\|\bT Q_1 - \bT Q_2\|_s \le \|Q_1 - Q_2\|_s
	\end{equation} 	
	where 
	$$
	\|Q\|_s \triangleq \max_{s,a,o} Q(s,a,o) - 
	\min_{s,a,o} Q(s,a,o).
	$$
\end{lem}	
\begin{proof}
	We start off by noting the trivial inequalities
	\begin{equation}
	\label{trivial}
	\begin{aligned}
	&\max_{s',a',o'} \left(Q_1(s', a', o') - Q_2(s', a', o')\right) \\
	&\ge Q_1(s_+, a_+, o) - Q_2(s_+, a_+, o)\\ 
	&\ge \min_{s',a',o'} \left(Q_1(s', a', o') - Q_2(s', a', o')\right).	
	\end{aligned}
	\end{equation}
	Also, let 
	$$
	o_i = \arg \min_{o    \in O} R(s,a,o) + Q_i(s_+, \pi(s_+), o)
	$$
	and
	$$
	a_i = \arg \max_{a    \in A} Q_i(s, a, o_j), ~~~~ i\neq j.
	$$
	The definition of the span semi-norm implies that
	\begin{equation}
	\label{spannormproof}
	\begin{aligned}
	&\|\bT Q_1 - \bT Q_2\|_s =\\
	&=\left\|\sum_{s_+} P(s,a,s_+)\left(\max_{\pi\in \Pi} \min_{o    \in O} (R(s,a,o) +  \E(Q_1(s_+, \pi(s_+), o)))-\right. \right. \\
	& \hspace{2cm} - \left.\left. \max_{\pi\in \Pi} \min_{o    \in O} (R(s,a,o) +  \E(Q_2(s_+, \pi(s_+), o)))\right)\right \|_s\\
	&= \max_{s,a,o} \sum_{s_+} P(s,a,s_+)\left(\max_{\pi\in \Pi} \min_{o    \in O} (R(s,a,o) +  \E(Q_1(s_+, \pi(s_+), o)))-\right.  \\
	& \hspace{2cm} - \left. \max_{\pi\in \Pi} \min_{o    \in O} (R(s,a,o) +  \E(Q_2(s_+, \pi(s_+), o)))\right) \\
	&~~~-\min_{s,a,o}  \sum_{s_+} P(s,a,s_+)\left(\max_{\pi\in \Pi} \min_{o    \in O} (R(s,a,o) +  \E(Q_1(s_+, \pi(s_+), o)))-\right. \\
	& \hspace{2cm} - \left. \max_{\pi\in \Pi} \min_{o    \in O} (R(s,a,o) +  \E(Q_2(s_+, \pi(s_+), o)))\right) \\
	&\le \max_{s,a,o} \sum_{s_+} P(s,a,s_+)\left(\max_{\pi\in \Pi} (R(s,a,o_2) +  \E(Q_1(s_+, \pi(s_+), o_2)))-\right.  \\
	& \hspace{2cm} - \left. \max_{\pi\in \Pi} (R(s,a,o_2) +  \E(Q_2(s_+, \pi(s_+), o_2)))\right) \\
	&~~~-\min_{s,a,o}  \sum_{s_+} P(s,a,s_+)\left(\max_{\pi\in \Pi} %
	(R(s,a,o_1) +  \E(Q_1(s_+, \pi(s_+), o_1)))-\right. \\
	& \hspace{2cm} - \left. \max_{\pi\in \Pi} %
	(R(s,a,o_1) +  \E(Q_2(s_+, \pi(s_+), o_1)))\right) \\
	&\le\max_{s,a,o} \sum_{s_+} P(s,a,s_+)\left(Q_1(s_+, a_1, o_2))-Q_2(s_+, a_1, o_2))\right) \\
	&~~~-\min_{s,a,o}  \sum_{s_+} P(s,a,s_+)\left(Q_1(s_+, a_2, o_1))-Q_2(s_+, a_2, o_1)) \right) \\
	&\le \max_{s,a,o} \sum_{s_+} P(s,a,s_+)\times \max_{s',a',o'} \left(Q_1(s', a', o') - Q_2(s', a', o')\right)\\
	&~~~ -\min_{s,a,o} \sum_{s_+} P(s,a,s_+)\times \min_{s',a',o'} \left(Q_1(s', a', o') - Q_2(s', a', o')\right)\\
	&= \max_{s',a',o'} \left(Q_1(s', a', o') - Q_2(s', a', o')\right)- \min_{s',a',o'} \left(Q_1(s', a', o') - Q_2(s', a', o')\right)\\
	&= \|Q_1 - Q_2\|_s.
	\end{aligned}
	\end{equation} 	
\end{proof}	
For convenience, let $\textup{e}: (s,a,o)\mapsto 1$ be a constant tensor with all elements equal to 1.
\begin{lem}
	\label{Qgas}
	Let $f\in \Phi$ be given, where the set $\Phi$ is defined as in Definition \ref{phi} and let $$\bT'(Q) = \bT(Q) - f(Q)\cdot \textup{e}$$
	The ordinary differential equation (ODE) 
	\begin{equation}
	\label{qdash}
	\dot{Q}(t) = \bT'(Q(t)) - Q(t)	%
	\end{equation}
	has a unique globally asymptotically stable equilibrium $Q^\star$, with $f(Q^\star) = v^\star$, where $Q^\star$ and $v^\star$ satisfy (\ref{Qrewardopt}).	
\end{lem}
\begin{proof}
	Introduce the operator
	\begin{equation*}
	\begin{aligned}
	\widehat{\bT}(Q) &= \bT(Q) - v\cdot \text{e}.
	\end{aligned}
	\end{equation*}
	According to lemma \ref{contraction}, we have that
	$$\|\bT Q_1 - \bT Q_2\|_\infty \le \|Q_1 - Q_2\|_\infty$$
	and hence, $\bT$ is Lipschitz. It's easy to verify that 
	$$\widehat{\bT}(Q_1) - \widehat{\bT}(Q_2) = \bT(Q_1) - \bT(Q_2)$$
	and therefore
	\begin{equation*}
	\begin{aligned}
	\|\widehat{\bT}(Q_1) - \widehat{\bT}(Q_2)\|_\infty 
	&\le \|Q_1 - Q_2\|_\infty,\\
	\|\widehat{\bT}(Q_1) - \widehat{\bT}(Q_2)\|_s\hspace{1.5mm} 
	&\le \|Q_1 - Q_2\|_s.
	\end{aligned}
	\end{equation*}
	Now consider the ODE:s
	\begin{equation}
	\label{qhat}
	\dot{Q}(t) = \widehat{\bT}(Q(t)) - Q(t)	
	\end{equation}
	and
	\begin{equation}
	\label{qdash}
	\dot{Q}(t) = \bT'(Q(t)) - Q(t)	= \widehat{\bT}(Q(t)) + (v - f(Q))\cdot \text{e}.
	\end{equation}
	Note that since $\bT$ and $f$ are Lipschitz, the ODE:s (\ref{qhat}) and (\ref{qdash}) are well posed.
	
	Since $\bT$ is Lipschitz and span semi-norm, the rest of the proof becomes identical to Theorem 3.4 along with Lemma 3.1, 3.2, and 3.3 in  \citep{abounadi:2001} and hence omitted here. 
\end{proof}

\begin{prp}[Borkar \& Meyn, 2000: Theorem 2.5]
	\label{borkar1}
	Consider the asynchronous algorithm given by 
	$$
	Q_{k+1} = Q_k + \alpha_k(h(Q_k) + M_{k+1})
	$$
	where $\alpha_k(s,a,o) = 1_{(s, a, o)}(s_k, a_k, o_k)\times \beta_{N(k,s,a,o)}$. 
	Suppose that 
	\begin{enumerate}
		\item $M_k$ is a martingale sequence with respect to 
		the sigma algebra $\mathcal{F}_k  = \sigma(Q_t, M_t, t\le k)$, that is
		$$
		\E(M_{k+1} \mid \mathcal{F}_k) = 0
		$$
		and that there exists a constant $C_1>0$ such that
		$$
		\E(\|M_{k+1}\|^2\mid \mathcal{F}_k)\le 
		C_1(1+\|Q_k\|^2).
		$$
		\item Assumptions \ref{lr} and \ref{ou} hold. 
		\item The limit
		$$
		h_\infty(X) = \lim_{z\rightarrow \infty} \frac{h(zX)}{z} %
		$$  
		exists.
		\item $\dot{Q}(t) = h(Q(t))$ has a unique globally asymptotically stable equilibrium $Q^\star$.
	\end{enumerate}
	Then, $Q_k \rightarrow Q^\star$ with probability 1 as $k\rightarrow \infty$ for any initial value $Q(0)$.
\end{prp}

\begin{proof}[Proof of Theorem \ref{algo2}]
	Introduce the operator
	{
		\begin{equation*}
		\begin{aligned}
		(\mathbf{T}Q)(s,a,o) %
		&= \sum_{s_+} P(s,a,s_+)\max_{\pi\in \Pi} \min_{o\in O}\left(R(s,a,o) + \E(Q(s_+, \pi(s_+), o))\right). \\
		\end{aligned}
		\end{equation*}
	}
	For convenience, let %
	$$\alpha_k(s,a,o) = 1_{(s, a, o)}(s_k, a_k, o_k)\cdot \beta_{N(k,s,a,o)},$$ 
	{
		$$
		M_{k+1}(s,a,o) = \max_{\pi\in \Pi} \min_{o\in O}(R(s,a,o) + \E(Q_k(s_{k+1}, \pi(s_{k+1}), o))) - (\bT Q_k)(s,a,o), 
		$$
	}
	and 
	$$
	h(Q) = \bT Q- f(Q)\cdot\text{e} - Q.
	$$
	Then, 
	$$
	Q_{k+1} = Q_k + \alpha_k( h(Q_k) + M_{k+1}).
	$$	
	We will now show that conditions 1 - 4 in Proposition \ref{borkar1} hold, and therefore $Q_k\rightarrow Q^\star$ with probability 1, where $Q^\star$ is the solution to (\ref{Qrewardopt}).
	\begin{enumerate}
		\item Let $\mathcal{F}_k$ be the sigma algebra $\sigma(Q_t, M_t, t\le k)$. Clearly,
		$$
		\E(M_{k+1} \mid \mathcal{F}_k) = 0
		$$
		and
		$$
		\E(\|M_{k+1}\|^2\mid \mathcal{F}_k)\le 
		C_1(1+\|Q_k\|^2)
		$$
		for some constant $C_1>0$. 
		\item We have supposed that assumptions \ref{lr} and \ref{ou} hold.
		\item Let $h(X) = \bT(X) - X - f(X) \cdot \text{e}$ and introduce
		\begin{equation}
		\label{Tbar}
		\begin{aligned}
		(\bar{\bT}Q)(s,a,o) %
		&= \max_{a_+\in A} \sum_{s_+} P(s,a,s_+)Q(s_+, a_+, o). 
		\end{aligned}
		\end{equation}
		Then, the limit
		\begin{equation*}
		\begin{aligned}
		h_\infty(X) &= \lim_{z\rightarrow \infty} h(zX)/z \\
		&= \bar{\bT}(X) -X - f(X) \cdot \text{e}
		\end{aligned}
		\end{equation*}
		exists. 
		\item By noting that 
		$$h(x) = \bT(X) - X - f(X) \cdot \text{e} = \bT'(X) - X$$
		we can apply Lemma \ref{Qgas} and conclude that $\dot{Q}(t) = h(Q(t))$ has a unique globally asymptotically stable equilibrium $Q^\star$.
	\end{enumerate}
	Thus, according to Proposition \ref{borkar1}, the iterators $Q_k$
	in (\ref{optQ}) converge to $Q^\star$, where $h(Q^\star)=0$ and hence the unique solution to (\ref{Qrewardopt}).
	Thus, the policy $\pi^\star \in \Pi$ given by 
	$$\pi^\star(s) = \argmax_{\pi} \min_{o\in O} Q^\star(s, \pi(s), o)$$
	maximizes (\ref{minreward2}), and the proof is complete.
\end{proof}

\section{Proof of Theorem \ref{0sumgame}}
Let
$$
\mathcal{L}(\pi, j) = \E \left(\sum_{k=0}^{\infty}\gamma^k r^j(s_k, \pi(s_k))\right).
$$
Consider the zero-sum game
\begin{equation*}
\max_{\pi\in\Pi}  \min_{j \in [J]} \mathcal{L}(\pi, j).
\end{equation*} 
Suppose that $\pi$ is a policy such that 
\begin{equation*}
\E \left(\sum_{k=0}^{\infty}\gamma^k r^j(s_k, \pi(s_k))\right) < 0
\end{equation*}
for some $j$. Then, 
$$
\mathcal{L}(\pi, j) < 0
$$
which implies
$$
\min_{j \in [J]}\mathcal{L}(\pi, j) < 0.
$$
Thus, if
\begin{equation*}
\label{maxminLambda}
\max_{\pi\in\Pi}  \min_{j \in [J]} \mathcal{L}(\pi, j) \ge 0
\end{equation*}
then, there must exist a policy $\pi$ that satisfies
\begin{equation}
\E \left(\sum_{k=0}^{\infty}\gamma^k r^j(s_k, \pi(s_k))\right) \ge 0
\end{equation}
for all $j$, and we get
$$
\min_{j \in [J]} \mathcal{L}(\pi, j) \ge 0.
$$
On the other hand, suppose that
\begin{equation*}
\max_{\pi\in\Pi} \min_{j \in [J]} \mathcal{L}(\pi, j) < 0.
\end{equation*} 
Then, there doesn't exist a policy $\pi$ such that
\begin{equation*}
\E \left(\sum_{k=0}^{\infty}\gamma^k r^j(s_k, \pi(s_k))\right) \ge 0
\end{equation*}
for all $j$, because it would imply that 
\begin{equation*}
\max_{\pi\in\Pi} \min_{j \in [J]} \mathcal{L}(\pi, j) \ge 0
\end{equation*} 
which is a contradiction, and the proof is complete.

\section{Proof of Theorem \ref{0sumgame2}}
Let
\begin{equation*}
\begin{aligned}
\mathcal{L}(\pi, j) %
& = \lim_{T\rightarrow \infty} \E \left(\frac{1}{T}\sum_{k=0}^{T-1} r^j(s_k, \pi(s_k))\right)
\end{aligned}
\end{equation*}
where the expectation is taken over $s_k$ and $\pi$. 
The rest of the proof is similar to the proof of Theorem \ref{0sumgame}.
\end{document}